\documentclass{amsart}

\newtheorem{theorem}{Theorem}[section]
\newtheorem{lemma}[theorem]{Lemma}

\newtheorem{proposition}[theorem]{Proposition}
\newtheorem{corollary}[theorem]{Corollary}

\theoremstyle{definition}

\theoremstyle{remark}
\newtheorem{remark}[theorem]{Remark}

\numberwithin{equation}{section}

\usepackage[pdftex,bookmarks=true]{hyperref}
\usepackage{amscd}
      \usepackage[nohug,heads=littlevee]{diagrams}
      \diagramstyle[labelstyle=\scriptstyle]

\newcommand{\isomorphic}{\simeq}

\newcommand{\liefont}[1]{\mathfrak{#1}}

\newcommand{\s}[3]{s^{#1}_{#2#3}}

\begin{document}

\title{Generic Representation Theory of the Unipotent Upper Triangular Groups}

\author{Michael Crumley}
\address{Department of Mathematics, The University of Toledo,
Toledo, Ohio 43606}
\email{mikecrumley@hotmail.com}


\date{October 2010.}



\begin{abstract}
It is generally believed (and for the most part is probably true)
that Lie theory, in contrast to the characteristic zero case, is
insufficient to tackle the representation theory of algebraic
groups over prime characteristic fields.  However, in this paper
we show that, for a large and important class of unipotent
algebraic groups (namely the unipotent upper triangular groups
$U_n$), and under a certain hypothesis relating the characteristic
$p$ to both $n$ and the dimension $d$ of a representation
(specifically, $p \geq \text{max}(n,2d)$, Lie theory is completely
sufficient to determine the representation theories of these
groups. To finish, we mention some important analogies (both
functorial and cohomological) between the characteristic zero
theories of these groups and their `generic' representation theory
in characteristic $p$.

\end{abstract}

\maketitle


\section{Introduction}

In this paper we extend a result for representations of the
Additive group $G_a$ given in \cite{SFB} and the Heisenberg group
$U_3$ given in \cite{myPaper}.  Namely, we give an intimate
connection between the characteristic zero representation theory
of the unipotent upper triangular groups $U_n$, and their
characteristic $p >0$ theory when $p$ greater than or equal to
both $n$ and twice the dimension of a representation.  In
particular, such representations are always given rise to by a
commuting product of Lie algebra representations, one Lie algebra
representation for each of its Frobenius layers, and conversely
and such collection of Lie algebra representations induces a
representation of $U_n$ in this fashion.

The essence of what we will be proving in the present paper is
best illustrated by the obvious analogy between the following two
theorems.  The first is very well known, the second not quite so
much.

\begin{theorem}Let $k$ be a field of characteristic zero.  Every representation of the Heisenberg group (i.e.~$U_3$) over $k$
is of the form
\[ e^{xX+yY+(z-xy/2)Z}\]
where $X,Y$ and $Z$ are nilpotent matrices over $k$ satisfying $Z
= [X,Y]$ and $[Z,X] = [Z,Y] = 0$.  Further, any such collection
$X,Y,Z$ gives a representation of $U_3$ over $k$ according to this
formula.
\end{theorem}

\begin{theorem} (see theorem 1.3 of \cite{myPaper})
\label{heisenbergGroupTheorem} Let $k$ be a field of
characteristic $p$, and suppose $p \geq 2d$. Then every
$d$-dimensional representation of the Heisenberg group over $k$ is
of the form
\begin{equation*}
\begin{split}
 &e^{xX_0+yY_0 + (z-xy/2)Z_0} e^{x^p X_1 + y^p Y_1 + (z^p-x^p
y^p/2)Z_1}\\ &\ldots e^{x^{p^m} X_m + y^{p^m} Y_m + (z^{p^m} -
x^{p^m} y^{p^m}/2)Z_m}
\end{split}
\end{equation*}
where $X_0,Y_0,Z_0,X_1,Y_1,Z_1\ldots,X_m, Y_m,Z_m$ is a collection
of $d \times d$ nilpotent matrices over $k$ satisfying
\begin{enumerate}
\item{$[X_i,Y_i] = Z_i$ and $[Z_i,X_i] = [Z_i,Y_i] = 0$ for every
$i$} \item{whenever $i \neq j$, $X_i,Y_i$ and $Z_i$ commute with
all of $X_j,Y_j$ and $Z_j$}
\end{enumerate}
Further, any such collection of $d \times d$ matrices gives a
representation of the Heisenberg group over $k$ according to the
above formula.
\end{theorem}

The morale: for fixed $d$ and large enough $p$, the
$d$-dimensional representation theory of $U_3$ over a field of
characteristic $p$ is in perfect analogy with the representation
theory of $U_3^\infty$ over a field of characteristic zero
(countable infinite product of copies of $U_3$).  Further, this
analogy is functorial; the characterization of morphisms between
modules over $U_3$ in characteristic zero carries over to a
characterization of morphisms for $U_3$ over a field of
characteristic $p$, assuming that both modules are of dimension $d
\geq p/2$.  Finally, this analogy is cohomological, and leads to
some interesting generic cohomology results, in a sense to be
discussed later.

 Let $k$ be a field, and let $U_n$ denote the
space of all $n \times n$ upper triangular unipotent matrices over
$k$, i.e.~those of the form
\[
\left(%
\begin{array}{ccccc}
  1 & x_{12} & x_{13} & \ldots & x_{1n} \\
   & 1 & x_{23} & \ldots & x_{2n} \\
   &  & \ddots & \ddots & \vdots \\
   &  &  & 1 & x_{n-1,n} \\
   &  &  &  & 1 \\
\end{array}%
\right)
\]
Let $\liefont{u}_n$ denote the Lie algebra of $U_n$, which we
identify as the space of all $n \times n$ strictly upper
triangular matrices over $k$, and likewise $\liefont{gl}_d$ the
Lie algebra of $GL_d$.  In what follows, if $M(x_1, \ldots, x_n)$
is a matrix whose entries are polynomials in the commuting
variables $x_1, \ldots, x_n$, then $M(x_1, \ldots, x_n)^{[m]}
\stackrel{\text{def}}{=} M(x_1^m, \ldots, x_n^m)$.  For example,
\[
\left(%
\begin{array}{cc}
  1 & 2x + 3y \\
  0 & 1 \\
\end{array}%
\right)^{[3]} = \left(
\begin{array}{cc}
  1 & 2x^3 + 3y^3 \\
  0 & 1 \\
\end{array}
\right)
\]
Note that, if $M(x_1, \ldots, x_n)$ is a representation of $U_n$
with corresponding comodule structure $V
\stackrel{\rho}{\rightarrow} V \otimes A$, then $M(x_1, \ldots,
x_n)^{[p]}$ has comodule structure given by the composition $V
\stackrel{\rho}{\rightarrow} V \otimes A \stackrel{1 \otimes
[p]}{\rightarrow} V \otimes A$, where $[p]$ denotes the linear map
$A \rightarrow A$ which sends each monomial of $A$ to its $p^{th}$
power (we argue in lemma \ref{pDoesItsThingLemma} that this new
map does indeed give a representation of $U_n$).  Note that $[p]$
 is not quite the same as a Frobenius twist; we are raising the
powers of the variables only, not the scalars. We can now state
the main theorem of this paper.

\begin{theorem}
\label{TheMainTheorem} Let $k$ be a field of characteristic $p>0$,
and suppose that $p \geq \text{max}(n,2d)$.

\begin{enumerate}

\begin{item}
let $\phi_0,\phi_1,\ldots,\phi_m:\liefont{u}_n \rightarrow
\liefont{gl}_d$ be a collection of Lie algebra homomorphisms such
that

\begin{enumerate}

\item{$\phi_i(X)$ is a nilpotent matrix for all $X \in
\liefont{u}_n$ and $0 \leq i \leq m$}

\item{For all $i \neq j$ and $X,Y \in \liefont{u}_n$, $\phi_i(X)$
commutes with $\phi_j(Y)$}

\end{enumerate}

Then the formula
\[ \Phi(g) = e^{\phi(\log(g))} {e^{\phi_1(\log(g))}}^{[p]} \ldots
{e^{\phi_m(\log(g))}}^{[p^m]} \] defines a valid $d$-dimensional
representation of $U_n$.

\end{item}

\begin{item}

Any valid $d$-dimensional representation of $U_n$ over $k$ is of
the form given by part (1).

\end{item}

\end{enumerate}

\end{theorem}

We ask the reader to note that this only a theorem about $p
>> n \text{ and } \text{dimension}$; for a fixed prime $p$ and large enough
dimension, the analogy between characteristic zero and
characteristic $p$ completely breaks down.  We direct the
interested reader to section 4 of \cite{myPaper} for an example of
a 10 dimensional representation of $U_3$ over the field
$\mathbb{Z}_2$ which is not of the above given form.

We shall prove the second part of our main theorem first, which
shall amount to checking that certain matrices associated to a
given representation satisfy certain relations.  We shall need
some terminology. Denote by $A_n$ the representing Hopf algebra of
$U_n$, which we identify as
\[ A_n = k[x_{ij}: 1 \leq i < j \leq n] \]
\[ \Delta: x_{ij} \mapsto 1 \otimes x_{ij} + \sum_{k=i+1}^{j-1}
x_{ik} \otimes x_{kj} + x_{ij} \otimes 1 \]
\[ \varepsilon:x_{ij} \mapsto 0 \]
We view a representation of $U_n$ on a $k$-vector space $V$ as a
comodule over its representing Hopf algebra $A_n$, (see section
3.2 of \cite{waterhouse} or chapter 2 of \cite{hopfalgebras}),
i.e.~as a $k$-linear map $\rho:V \rightarrow V \otimes A_n$
satisfying the diagrams
\begin{equation}
\label{comodDiagram1}
\begin{diagram}
V & \rTo^\rho & V \otimes A_n \\
\dTo^\rho & & \dTo_{1 \otimes \Delta} \\
V \otimes A_n & \rTo_{\rho \otimes 1} & V \otimes A_n \otimes A_n \\
\end{diagram}
\end{equation}
\begin{equation}
\label{comodDiagram2}
\begin{diagram}
V & \rTo^\rho & V \otimes A_n \\
& \rdTo_{\isomorphic} & \dTo_{1 \otimes \varepsilon} \\
& & V \otimes k \\
\end{diagram}
\end{equation}
If we fix a basis $\{e_1, \ldots, e_m\}$ for $V$, then we can
write $\rho:e_j \mapsto \sum_i e_i \otimes a_{ij}$, where
$(a_{ij})$ is the matrix formula for the representation in this
basis. Then the diagrams above are, in equation form
\begin{equation}
\label{comodEquation1} \Delta(a_{ij}) = \sum_k a_{ik} \otimes
a_{kj}
\end{equation}
\begin{equation}
\label{comodEquation2} \varepsilon(a_{ij}) = \delta_{ij}
\end{equation}

For an $n \times n$ strictly upper triangular matrix $M =
(m_{ij})$ with non-negative integer entries, let $x^M$ denote the
monomial expression
\[ x_{12}^{m_{12}} \ldots x_{1n}^{m_{1n}} x_{23}^{m_23} \ldots
x_{2n}^{m_{2n}} \ldots x_{n-1,n}^{m_{n-1,n}} \] Then the matrix
$(a_{ij})$ can be written uniquely as
\[ (a_{ij}) = \sum_M \chi(M) x^M \]
where the summation runs over all strictly upper triangular
matrices $M$ with non-negative integer entries, and $\chi(M)$ is
an $m \times m$ matrix with entries in $k$.

Example: consider the representation of $U_3$ with matrix formula
\[
\left(%
\begin{array}{cccccc}
  1 & 2x_{12} & x_{12} & 2x_{12}^2 & x_{13} & 2x_{12}x_{13} \\
   & 1 & 0 & x_{12} & 0 & x_{13} \\
   &  & 1 & 2x_{12} & x_{23} & 2x_{12}x_{23} \\
   &  &  & 1 & 0 & x_{23} \\
   &  &  &  & 1 & 2x_{12} \\
   &  &  &  &  & 1 \\
\end{array}%
\right)
\]
Let $M$ be the matrix
\[
\left(%
\begin{array}{ccc}
  0 & 1 & 1 \\
  0 & 0 & 0 \\
  0 & 0 & 0 \\
\end{array}%
\right)
\]
Then $x^M$ is the monomial $x_{12}x_{13}$, and to compute
$\chi(M)$ we ask, what is the `matrix of coefficients' of this
monomial in the above matrix formula?  That is
\[ \chi(M) =
\left(%
\begin{array}{cccccc}
  0 & 0 & 0 & 0 & 0 & 2 \\
   & 0 & 0 & 0 & 0 & 0 \\
   &  & 0 & 0 & 0 & 0 \\
   &  &  & 0 & 0 & 0 \\
   &  &  &  & 0 & 0 \\
   &  &  &  &  & 0 \\
\end{array}%
\right)
\]
Denote by $\varepsilon_{ij}$ the $n \times n$ matrix with a $1$ in
the $(i,j)^\text{th}$ entry, zeroes elsewhere.  Then the above
representation also gives
\[ \chi(\epsilon_{12})
=
\left(%
\begin{array}{cccccc}
  0 & 2 & 1 & 0 & 0 & 0 \\
   & 0 & 0 & 1 & 0 & 0 \\
   &  & 0 & 2 & 0 & 0 \\
   &  &  & 0 & 0 & 0 \\
   &  &  &  & 0 & 2 \\
   &  &  &  &  & 0 \\
\end{array}%
\right) \text{ and } \chi(2\varepsilon_{12}) =
\left(%
\begin{array}{cccccc}
  0 & 0 & 0 & 2 & 0 & 0 \\
   & 0 & 0 & 0 & 0 & 0 \\
   &  & 0 & 0 & 0 & 0 \\
   &  &  & 0 & 0 & 0 \\
   &  &  &  & 0 & 0 \\
   &  &  &  &  & 0 \\
\end{array}%
\right)
\]
and $\chi(2\varepsilon_{12} + \varepsilon_{13}) = 0$, since the
monomial $x_{12}^2 x_{13}$ never occurs in the representation.  We
must of course have $\chi(M) = 0$ for all but finitely many $M$,
since only finitely many monomials can occur in any given
representation.

In what follows $[X,Y]$ denotes the usual matrix Lie bracket
$XY-YX$.  The following implies part (2) of our main theorem, and
is what we will be proving in section \ref{necessitySection}.

\begin{theorem}
\label{TheMainTheoremSub} Let $k$ be a field of characteristic
$p>0$, $n
> 0$. If $p \geq 2d$, then every $d$-dimensional representation
$(V,\rho)$ of $U_n$ over $k$ satisfies, for all $m,n \geq 0$, and
all $1 \leq r < s \leq n$, $1 \leq t < u \leq n$:

\begin{enumerate}

\item{$\chi(p^m \varepsilon_{rs})$ is a nilpotent matrix}

\item{
\[
[\chi(p^m \varepsilon_{rs}),\chi(p^n \varepsilon_{tu})] = \left\{
\begin{array}{cc}
  0 & \hspace{.3cm}\text{ \emph{if }} m \neq n\\
  0 & \hspace{.3cm}\text{ \emph{if }} [\varepsilon_{rs},\varepsilon_{tu}] = 0 \\
  \chi(p^m[\varepsilon_{rs},\varepsilon_{tu}]) & \hspace{.3cm}\text{ \emph{ otherwise
  }}\\
\end{array}
\right.
\]
}

\end{enumerate}

\end{theorem}

\begin{proposition}
Theorem \ref{TheMainTheoremSub} implies part (2) of theorem
\ref{TheMainTheorem}.

\end{proposition}

\begin{proof}
For each $l$ and $1 \leq r < s \leq n$, set
\[ \phi_l(\varepsilon_{rs}) = \chi(p^l \varepsilon_{rs}) \]
and extend each $\phi_l$ linearly to all of $\liefont{u}_n$.  Then
for fixed $l$, the previous theorem gives
\[ [\phi_l(\varepsilon_{rs}),\phi_l(\varepsilon_{tu}) ] =
\phi_l([\varepsilon_{rs},\varepsilon_{tu}]) \] which says that
each $\phi_l$ is a Lie algebra homomorphism.  For $m \neq l$, the
previous theorem gives
\[ [\phi_l(\varepsilon_{rs}), \phi_m(\varepsilon_{tu}) ] =
[\chi(p^l \varepsilon_{rs}),\chi(p^m \varepsilon_{tu}) ] = 0 \]
which is predicted by part (2) of theorem \ref{TheMainTheorem}.
Finally,
 assuming that the formula
\[ \Phi(g) = e^{\phi(\log(g))} {e^{\phi_1(\log(g))}}^{[p]} \ldots
{e^{\phi_m(\log(g))}}^{[p^m]} \] actually is a representation of
$U_n$ (which will be proven in section \ref{sufficiencySection}),
we want to see that it is the same representation as $(a_{ij}) =
\sum_M \chi(M) x^M$.  Lemmas \ref{UnDecompositionLemma} and
\ref{UntoGaLemma} given in section \ref{necessitySection} make it
clear that it suffices to check that, for each $r,s$ and $l$, the
matrix of coefficients of the monomial $x_{rs}^{p^l}$ in the first
formula is actually $\chi(p^l \varepsilon_{rs})$, the same as in
the second formula.

We prove this first for the case of $m=0$. Let $g = 1 + \sum_{1
\leq i < j \leq n} x_{ij}\varepsilon_{ij}$ be an arbitrary element
of $U_n$. Then we leave it to the reader to verify that
\begin{equation*}
\begin{split}
\log(g) &\stackrel{\text{def}}{=} \sum_{k=1}^{n-1}
\frac{(-1)^{k-1}}{k}(g-1)^k \\
&= \sum_{1 \leq i < j \leq n} x_{ij} \varepsilon_{ij} + \ldots \\
\end{split}
\end{equation*}
where $(\ldots)$ denotes a sum of matrices whose monomial
coefficients have length greater than $1$.  As $\phi_0$ is linear
we have
\[ \phi_0(\log(g)) = \sum_{1 \leq i < j \leq n} x_{ij}
\phi_0(\varepsilon_{ij}) + \phi_0(\ldots) \] Then
\begin{equation*}
\begin{split}
\Phi(g) &= e^{\phi_0(\log(g))} \\
&= 1 + \left(\sum_{1 \leq i < j \leq n} x_{ij}
\phi_0(\varepsilon_{ij})
+ \phi_0(\ldots)\right) + \ldots \\
\end{split}
\end{equation*}
where $(\ldots)$ denotes terms including higher powers of
$\phi_0(\log(g))$. Clearly then we have that the matrix of
coefficients of the terms $x_{ij}$ are exactly
$\phi_0(\varepsilon_{ij}) = \chi(\varepsilon_{ij})$.

For $m > 0$, consider
\begin{equation*}
\begin{split}
 \Phi(g) &= e^{\phi_0(\log(g))}{e^{\phi_1(\log(g))}}^{[p]} \ldots
{e^{\phi_l(\log(g))}}^{[p^l]} \ldots {e^{\phi_m(\log(g))}}^{[p^m]}
\\
&= \sum_{k_0=0}^d \ldots \sum_{k_l=0}^d \ldots \sum_{k_m=0}^d
\left(\frac{\phi_0(\log(g))^{k_0}}{k_0!}\right)
\dots\left(\frac{\phi_l(\log(g))^{k_l}}{k_l!}\right)^{[p^l]} \dots
\left(\frac{\phi_m(\log(g))^{k_m}}{k_m!}\right)^{[p^m]}
\end{split}
\end{equation*}
Since $d \leq p/2$, we see that the only contribution to the
monomial $x_{rs}^{p^l}$ comes when all $k_i$ are zero except $k_l
= 1$, and thus the matrix of coefficients of the monomial
$x_{rs}^{p^l}$ is exactly $\phi_l(\varepsilon_{rs})$.  Recalling
that $\phi_i(\varepsilon_{rs}) \stackrel{\text{def}}{=}
\chi(p^i\varepsilon_{rs})$, this proves that $\Phi(g) = (a_{ij})$.

\end{proof}

Our approach to proving theorem \ref{TheMainTheoremSub} is
partially inductive. Take for example the group $U_4$:
\[
\left(%
\begin{array}{cccc}
  1 & x_{12} & x_{13} & x_{14} \\
   & 1 & x_{23} & x_{24} \\
   &  & 1 & x_{34} \\
   &  &  & 1 \\
\end{array}%
\right)
\]
This group contains four conspicuous subgroups, namely those
matrices of the form
\[ \label{embeddings}
\left(%
\begin{array}{cccc}
  1 & x_{12} & x_{13} & 0 \\
   & 1 & x_{23} & 0 \\
   &  & 1 & 0 \\
   &  &  & 1 \\
\end{array}%
\right),
\left(%
\begin{array}{cccc}
  1 & 0 & 0 & 0 \\
   & 1 & x_{23} & x_{24} \\
   &  & 1 & x_{34} \\
   &  &  & 1 \\
\end{array}%
\right),
\left(%
\begin{array}{cccc}
  1 & x_{12} & x_{13} & x_{14} \\
   & 1 & 0 & 0 \\
   &  & 1 & 0 \\
   &  &  & 1 \\
\end{array}%
\right)\]

\[ \text{ and }
\left(%
\begin{array}{cccc}
  1 & 0 & 0 & x_{14} \\
   & 1 & 0 & x_{24} \\
   &  & 1 & x_{34} \\
   &  &  & 1 \\
\end{array}%
\right)
\]
which are isomorphic to, respectively, $U_3$, $U_3$, $G_a^3$, and
$G_a^3$ (three-fold direct product of the Additive group).  Since
all of these various embeddings are given by simply setting
certain variables to zero, by the main result given in
\cite{myPaper} for the Heisenberg group, and by the result given
in \cite{SFB} or \cite{MyDissertation} for products of the
Additive group, we immediately conclude by induction that, for
instance,
\[[\chi(p^m \varepsilon_{23}),\chi(p^m \varepsilon_{34})] =
\chi(p^m[\varepsilon_{24}]) \] since the variables $x_{23}$ and
$x_{34}$ already occur together in a subgroup isomorphic to $U_3$.
Thus, by induction, the only brackets $[\chi(p^m
\varepsilon_{rs}),\chi(p^n \varepsilon_{tu})]$ that need to be
checked are those that do \emph{not} occur together in one of the
above given subgroups, namely
\begin{enumerate}
\item{when $r=1$ and $u=4$ (i.e.~when the first variable is in the
top row, and the second is in the right-most column)} \item{when
$\varepsilon_{rs} = \varepsilon_{14}$, and when $t \neq 1$ and $u
\neq 4$ (i.e.~when the first variable is in the top right corner,
and the second variable is in neither the top row nor the
right-most column)}
\end{enumerate}
Restricting ourselves to these cases shall, as we will see,
drastically simplify our calculations.

\section{Combinatorics}

Our arguments for proving theorem \ref{TheMainTheoremSub} are
heavily combinatorial, and the notation necessary to prove our
main theorem can be at times confusing in the abstract, so as we
go along we shall provide examples, using the case $n=4$, i.e.~the
group $U_4$, as a template.

The learned reader may recognize that what we are really doing in
this section is working out certain facts about the multiplicative
structure of $\text{Dist}(U_n)$, the distribution algebra of $U_n$
(see chapter 7 of \cite{jantzen}), but we have no need for this
terminology, and no understanding of distribution algebras is
assumed on the part of the reader.

Let $n$ be arbitrary, and throughout, fix a $d$-dimensional
representation $(a_{ij})$ for $U_n$ over the field $k$, and write
\[ (a_{ij}) = \sum_M \chi(M) x^M \]
where the summation runs over all $n \times n$ strictly upper
triangular matrices $M = (m_{ij})$ with non-negative integer
entries, and $\chi(M)$ is a $d \times d$ matrix with entries in
$k$ for each $M$, and where we define
\[ x^M \stackrel{\text{defn}}{=} \prod_{1 \leq i < j \leq n}
x_{ij}^{m_{ij}} \] Define the following variable matrices $S_1,
\ldots, S_n$.  The non-zero entries of $S_k$ are written
$s_{ij}^k$.  $S_1$ and $S_2$ are demanded to be strictly upper
triangular, $S_3$ is strictly-strictly upper triangular, $S_4$ is
strictly-strictly-strictly upper triangular, $\ldots$, $S_n$ is
$(n-1) \times$ strictly upper triangular (i.e., $S_n$ has a
non-zero entry in its $(1,n)^{\text{th}}$ spot only). In other
words, for each $1 \leq i < j \leq n$ we have the variables
$s_{ij}^1$ and for each $2 \leq k \leq n$, and each $1 \leq i < j
\leq n$ with $j-i \geq k-1$, we have the variables $s_{ij}^k$.

Example: in the case of $n=4$, the matrices $S_1, S_2, S_3, S_4$
are given by
\begin{equation*}
\begin{array}{cc}
S_1 =
\left(%
\begin{array}{cccc}
  0 & \s{1}{1}{2} & \s{1}{1}{3} & \s{1}{1}{4} \\
   & 0 & \s{1}{2}{3} & \s{1}{2}{4} \\
   &   & 0 & \s{1}{3}{4} \\
   &   &   & 0 \\
\end{array}%
\right) & S_2 =
\left(%
\begin{array}{cccc}
  0 & \s{2}{1}{2} & \s{2}{1}{3} & \s{2}{1}{4} \\
   & 0 & \s{2}{2}{3} & \s{2}{2}{4} \\
   &   & 0 & \s{2}{3}{4} \\
   &   &   & 0 \\
\end{array}%
\right)
\\
 & \\
 S_3 =
\left(%
\begin{array}{cccc}
  0 & 0 & \s{3}{1}{3} & \s{3}{1}{4} \\
   & 0 & 0 & \s{3}{2}{4} \\
   &   & 0 & 0 \\
   &   &   & 0 \\
\end{array}%
\right) & S_4 = \left(
\begin{array}{cccc}
  0 & 0 & 0 & \s{4}{1}{4} \\
   & 0 & 0 & 0 \\
   &   & 0 & 0 \\
   &   &   & 0 \\
\end{array}%
\right)
\\
\end{array}
\end{equation*}

For $1 \leq i < j \leq n$, define the following variable
expressions among the $s_{ij}^k$:
\[ L_{ij} \stackrel{\text{def}}{=} \sum_{k=j}^n s_{ik}^{j-i+1}
\text{ and } R_{ij} \stackrel{\text{def}}{=} \sum_{k=1}^i
s_{k,j}^{i-k+1} \] For example, in the case of $n=4$, $L_{12} =
s_{12}^2 + s_{13}^2 + s_{14}^2$, and $R_{34} = s_{14}^3 + s_{24}^2
+ s_{34}^1$.

The following notation will be useful: if $B_1 = (b_{ij}^1),
\ldots, B_k = (b_{ij}^k), M = (m_{ij})$ are matrices with
non-negative integer entries such that $B_1 + \ldots + B_k = M$,
then the formal multinomial expression
\[{M \choose B_1, \ldots, B_k} \]
is shorthand for $\prod_{ij} {m_{ij} \choose b_{ij}^1, \ldots,
b_{ij}^k}$.  In other words, it is just the product of the
multinomial coefficients of the individual entries.  For example
\[
\left(
\left(%
\begin{array}{ccc}
  0 & 2 & 3 \\
  0 & 0 & 1 \\
  0 & 0 & 0 \\
\end{array}%
\right) \text{ choose }
\left(%
\begin{array}{ccc}
  0 & 1 & 2 \\
  0 & 0 & 0 \\
  0 & 0 & 0 \\
\end{array}%
\right),\left(%
\begin{array}{ccc}
  0 & 1 & 0 \\
  0 & 0 & 1 \\
  0 & 0 & 0 \\
\end{array}%
\right),\left(%
\begin{array}{ccc}
  0 & 0 & 1 \\
  0 & 0 & 0 \\
  0 & 0 & 0 \\
\end{array}%
\right) \right)
\]
equals
\[ {2 \choose 1,1,0}{1 \choose 0,1,0}{3 \choose 2,0,1} = 6 \]

\begin{proposition}
\label{UnBigEquationProp} In the notation above, the matrix
$(\Delta(a_{ij}))$ is equal to
\[ \sum_M \chi(M) \left[ \sum_{S_1 + \ldots + S_n = M} {S_1 +
\ldots + S_n \choose S_1, \ldots, S_n} \left( \prod_{1 \leq i < j
\leq n} x_{ij}^{L_{ij}} \right) \otimes \left( \prod_{1 \leq i < j
\leq n} x_{ij}^{R_{ij}} \right) \right] \] where the first
summation runs over all $n \times n$ strictly upper triangular
matrices $M$ with non-negative integer entries, and the second
runs over all $S_1 + S_2 + \ldots + S_n = M$ with non-negative
integer entries of the form defined in the above paragraph.
\end{proposition}

Before we prove the proposition, let us illustrate with an example
what it is actually saying, again using the case of $n=4$.  Using
$(a_{ij}) = \sum_M \chi(M) x^M$, we have
\begin{equation*}
\begin{split}
\Delta(a_{ij}) &= \sum_M \chi(M) \Delta(x^M) \\
&=\sum_M \chi(M)_{ij} \Delta(x_{12})^{m_{12}}
\Delta(x_{23})^{m_{23}} \Delta(x_{34})^{m_{34}}
\Delta(x_{13})^{m_{13}} \Delta(x_{24})^{m_{24}}
\Delta(x_{14})^{m_{14}} \\
& = \sum_M \chi(M)_{ij} (1 \otimes x_{12} + x_{12} \otimes
1)^{m_{12}}(1 \otimes x_{23} + x_{23} \otimes 1)^{m_{23}}(1
\otimes x_{34} + x_{34} \otimes 1)^{m_{34}} \\
& \quad \quad \quad \quad \quad \quad (1 \otimes x_{13} + x_{12}
\otimes x_{23} + x_{13} \otimes 1)^{m_{13}}(1 \otimes x_{24} +
x_{23} \otimes x_{34} + x_{24} \otimes 1)^{m_{24}} \\
& \quad \quad \quad \quad \quad \quad (1 \otimes x_{14} + x_{12}
\otimes x_{24} + x_{13} \otimes x_{34} + x_{14} \otimes
1)^{m_{14}} \\
\end{split}
\end{equation*}
and by repeated application of the binomial/multinomial theorem
\begin{equation*}
\begin{split}
= \sum_M & \chi(M) \\
& \left( \sum_{\s{1}{1}{2}+\s{2}{1}{2} = m_{12}}{m_{12} \choose \s{1}{1}{2},\s{2}{1}{2}} (1 \otimes x_{12})^{s_{12}^1} (x_{12} \otimes 1)^{s_{12}^2}         \right) \\
& \left( \sum_{\s{1}{2}{3}+\s{2}{2}{3} = m_{23}} {m_{23} \choose \s{1}{2}{3},\s{2}{2}{3}} (1 \otimes x_{23})^{s_{23}^1}(x_{23} \otimes 1)^{s_{23}^2}        \right) \\
& \left( \sum_{\s{1}{3}{4}+\s{2}{3}{4} = m_{24}} {m_{34} \choose \s{1}{3}{4},\s{2}{3}{4}} (1 \otimes x_{34})^{s_{34}^1} (x_{34} \otimes 1)^{s_{34}^2}          \right) \\
& \left( \sum_{\s{1}{1}{3} + \s{2}{1}{3} + \s{3}{1}{3} = m_{13}}
{m_{13} \choose \s{1}{1}{3}, \s{2}{1}{3}, \s{3}{1}{3}} (1 \otimes
x_{13})^{s_{13}^1}(x_{12} \otimes x_{23})^{s_{13}^2}(x_{13}
\otimes 1)^{s_{13}^3}\right)
\\
& \left( \sum_{\s{1}{2}{4} + \s{2}{2}{4} + \s{3}{2}{4} = m_{24}}
{m_{24} \choose \s{1}{2}{4}, \s{2}{2}{4}, \s{3}{2}{4}  } (1
\otimes x_{24})^{s_{24}^1}(x_{23} \otimes x_{34})^{s_{24}^2}
(x_{24} \otimes 1)^{s_{24}^3} \right)
\\
& \left( \sum_{\s{1}{1}{4} + \s{2}{1}{4} + \s{3}{1}{4} +
\s{4}{1}{4} = m_{14}} {m_{14} \choose \s{1}{1}{4}, \s{2}{1}{4},
\s{3}{1}{4}, \s{4}{1}{4}  } (1 \otimes x_{14})^{s_{14}^1}(x_{12} \otimes x_{24})^{s_{14}^2} (x_{13} \otimes x_{34})^{s_{14}^3} (x_{14} \otimes 1)^{s_{14}^4}        \right) \\
\end{split}
\end{equation*}
and, condensing
\begin{equation*}
\begin{split}
= \sum_M & \chi(M) \\
& \left( \sum_{\s{1}{1}{2}+\s{2}{1}{2} = m_{12}}{m_{12} \choose \s{1}{1}{2},\s{2}{1}{2}} x_{12}^{\s{2}{1}{2}} \otimes x_{12}^{\s{1}{1}{2}}  \right) \\
& \left( \sum_{\s{1}{2}{3}+\s{2}{2}{3} = m_{23}} {m_{23} \choose \s{1}{2}{3},\s{2}{2}{3}} x_{23}^{\s{2}{2}{3}} \otimes x_{23}^{\s{1}{2}{3}} \right) \\
& \left( \sum_{\s{1}{3}{4}+\s{2}{3}{4} = m_{24}} {m_{34} \choose \s{1}{3}{4},\s{2}{3}{4}} x_{34}^{\s{2}{3}{4}} \otimes x_{34}^{\s{1}{3}{4}} \right) \\
& \left( \sum_{\s{1}{1}{3} + \s{2}{1}{3} + \s{3}{1}{3} = m_{13}}
{m_{13} \choose \s{1}{1}{3}, \s{2}{1}{3}, \s{3}{1}{3}}
x_{12}^{\s{2}{1}{3}} x_{13}^{\s{3}{1}{3}} \otimes
x_{13}^{\s{1}{1}{3}} x_{23}^{\s{2}{1}{3} } \right)
\\
& \left( \sum_{\s{1}{2}{4} + \s{2}{2}{4} + \s{3}{2}{4} = m_{24}}
{m_{24} \choose \s{1}{2}{4}, \s{2}{2}{4}, \s{3}{2}{4}  }
x_{23}^{\s{2}{2}{4}} x_{24}^{\s{3}{2}{4}} \otimes
x_{24}^{\s{1}{2}{4}} x_{34}^{\s{2}{2}{4}}\right)
\\
& \left( \sum_{\s{1}{1}{4} + \s{2}{1}{4} + \s{3}{1}{4} +
\s{4}{1}{4} = m_{14}} {m_{14} \choose \s{1}{1}{4}, \s{2}{1}{4},
\s{3}{1}{4}, \s{4}{1}{4}  } x_{12}^{\s{2}{1}{4}}
x_{13}^{\s{3}{1}{4}} x_{14}^{\s{4}{1}{4}} \otimes
x_{14}^{\s{1}{1}{4}}
x_{24}^{\s{2}{1}{4}} x_{34}^{\s{3}{1}{4}} \right) \\
\end{split}
\end{equation*}
With the variable matrices $S_1,S_2,S_3,S_4$ so previously
defined, we can collapse all of these summations to
\begin{equation*}
\label{U4deltaEquation2}
\begin{split}
 \sum_M \chi(M) & \left [\sum_{S_1 + S_2 + S_3 + S_4 = M} {M \choose S_1, S_2, S_3, S_4}\right.\\
& \quad \quad \quad
x_{12}^{\s{2}{1}{2}+\s{2}{1}{3}+\s{2}{1}{4}}x_{23}^{\s{2}{2}{3}+\s{2}{2}{4}}x_{34}^{\s{2}{3}{4}}
x_{13}^{\s{3}{1}{3} + \s{3}{1}{4}}
x_{24}^{\s{3}{2}{4}}x_{14}^{\s{4}{1}{4}} \\
& \quad \bigotimes \\
&\quad \quad \quad
x_{12}^{\s{1}{1}{2}}x_{23}^{\s{1}{2}{3}+\s{2}{1}{3}}
x_{34}^{\s{1}{3}{4}+\s{2}{2}{4}+\s{3}{1}{4}}
x_{13}^{\s{1}{1}{3}} x_{24}^{\s{1}{2}{4}+\s{2}{1}{4}} x_{14}^{\s{1}{1}{4}} \biggr] \\
\end{split}
\end{equation*}
which the reader can verify by hand can be written
\[ \sum_M \chi(M) \left[ \sum_{S_1 + S_2 + S_3 + S_4 = M} {S_1 + S_2 + S_3 + S_4 \choose S_1, S_2, S_3, S_4} \left( \prod_{1 \leq i <
j \leq 4} x_{ij}^{L_{ij}} \right) \otimes \left( \prod_{1 \leq i <
j \leq 4} x_{ij}^{R_{ij}} \right) \right] \]

We now prove proposition \ref{UnBigEquationProp} for arbitrary
$n$.

\begin{proof}
Compute: $(\Delta(a_{ij})) =$
\begin{equation*}
\begin{split}
&= \sum_M \chi(M) \Delta(x^M) \\
&= \sum_M \chi(M) \prod_{1 \leq i < j \leq n}
\Delta(x_{ij})^{m_{ij}} \\
&=\sum_M \chi(M) \prod_{1 \leq i < j \leq n} (1 \otimes x_{ij} +
\sum_{k=i+1}^{j-1} x_{ik} \otimes x_{kj} + x_{ij} \otimes
1)^{m_{ij}} \\
&= \sum_M \chi(M) \prod_{1 \leq i < j \leq n} \left(
\sum_{s_{ij}^1 + \ldots + s_{ij}^{j-i+1} = m_{ij}} {s_{ij}^1 +
\ldots +
s_{ij}^{j-i+1} \choose s_{ij}^1, \ldots ,s_{ij}^{j-i+1}} \right. \\
& \quad \quad \quad \quad \quad (1 \otimes x_{ij})^{s_{ij}^1}
\prod_{k=i+1}^{j-1} (x_{ik} \otimes x_{kj})^{s_{ij}^{k-i+1}}
(x_{ij} \otimes 1)^{s_{ij}^{j-i+1}} \biggr) \\
&= \sum_M \chi(M) \prod_{1 \leq i < j \leq n} \left(
\sum_{s_{ij}^1 + \ldots + s_{ij}^{j-i+1} = m_{ij}} {s_{ij}^1 +
\ldots +
s_{ij}^{j-i+1} \choose s_{ij}^1, \ldots ,s_{ij}^{j-i+1}} \right. \\
& \quad \quad \quad \quad \quad x_{i,i+1}^{s_{ij}^2}
x_{i,i+2}^{s_{ij}^3} \ldots x_{ij}^{s_{ij}^{j-i+1}} \otimes
x_{ij}^{s_{ij}^1} x_{i+1,j}^{s_{ij}^2} \ldots
x_{j-1,j}^{s_{ij}^{j-i}} \biggr) \\
&= \sum_M \chi(M) \hspace{-2cm} \sum_{{s_{12}^1 + s_{12}^2 =
m_{12}} \atop {{ s_{23}^1+s_{23}^2 = m_{23}  } \atop {{ \vdots }
\atop {{ s_{n-1,n}^1 + s_{n-1,n}^2 = m_{n-1,n}  } \atop {{
s_{13}^1 + s_{13}^2 + s_{13}^3 = m_{13} } \atop {{s_{24}^1 +
s_{24}^2 + s_{24}^3 = m_{24}} \atop {{ \vdots } \atop {{
s_{n-2,n}^1 + s_{n-2,n}^2 + s_{n-2,n}^3 = m_{n-2,n} } \atop {{
\vdots  } \atop {s_{1,n}^1 + s_{1,n}^2 + \ldots + s_{1,n}^n =
m_{1,n} }}}}}}}}} } \hspace{-1.5cm} \prod_{1 \leq i < j \leq n}
{s_{ij}^1 + \ldots + s_{ij}^{j-i+1} \choose s_{ij}^1, \ldots
,s_{ij}^{j-i+1}} x_{i,i+1}^{s_{ij}^2} x_{i,i+2}^{s_{ij}^3} \ldots
x_{ij}^{s_{ij}^{j-i+1}} \otimes x_{ij}^{s_{ij}^1}
x_{i+1,j}^{s_{ij}^2} \ldots x_{j-1,j}^{s_{ij}^{j-i}}
\\
&= \sum_M \chi(M) \sum_{S_1 + \ldots + S_n = M} {M \choose S_1,
\ldots, S_n} \prod_{1 \leq i < j \leq n} x_{i,i+1}^{s_{ij}^2}
x_{i,i+2}^{s_{ij}^3} \ldots x_{ij}^{s_{ij}^{j-i+1}} \otimes
x_{ij}^{s_{ij}^1}
x_{i+1,j}^{s_{ij}^2} \ldots x_{j-1,j}^{s_{ij}^{j-i}} \\
\end{split}
\end{equation*}
Our task now is to write the expression $\prod_{1 \leq i < j \leq
n} x_{i,i+1}^{s_{ij}^2} x_{i,i+2}^{s_{ij}^3} \ldots
x_{ij}^{s_{ij}^{j-i+1}} \otimes x_{ij}^{s_{ij}^1}
x_{i+1,j}^{s_{ij}^2} \ldots x_{j-1,j}^{s_{ij}^{j-i}}$ in the form
\[ \left( \prod_{1 \leq i < j \leq n} x_{ij}^{L_{ij}} \right)
\otimes \left( \prod_{1 \leq i < j \leq n} x_{ij}^{R_{ij}} \right)
\]
and to figure out what $L_{ij}$ and $R_{ij}$ are.  To compute
$L_{ij}$ we ask, in which factors, for $1 \leq u < k \leq n$, does
the term $x_{ij}$ occur in the left tensor slot, and what is its
exponent?  For this to happen, it is necessary and sufficient that
 $u=i$ and $j \leq k$, and in this case the exponent
of $x_{ij}$ is $s_{i,k}^{j-i+1}$.  Summing this expression over
all $k$ greater than or equal to $j$ gives
\[ L_{ij} = \sum_{k=j}^n s_{ik}^{j-i+1} \]
as claimed.  To compute $R_{ij}$ we ask, for which $1 \leq k < u
\leq n$ does $x_{ij}$ occur in the right tensor slot, and what is
its exponent?  For this to happen, it is necessary and sufficient
that $u = j$, and that $k \leq i$, and in this case the exponent
of $x_{ij}$ is $s_{k,j}^{i-k+1}$.  Summing this expression over
all $k$ less than or equal to $i$ gives
\[ R_{ij} = \sum_{k=1}^i s_{k,j}^{i-k+1} \]
as claimed.  This completes the proof.
\end{proof}

\begin{proposition}
\label{UnMiniProp} The matrix $(\sum_k a_{ik} \otimes a_{kj})$ is
equal to
\[ \sum_{M,N} \chi(M) \chi(N) x^M \otimes x^N \]
where the summation runs over all $n \times n$ strictly upper
triangular matrices with non-negative integer entries, and
$\chi(M) \chi(N)$ is the usual product of matrices.
\end{proposition}

\begin{proof}
We have that $a_{ik} = \sum_M \chi(M)_{ik} x^M$, and $a_{kj} =
\sum_N \chi(N)_{kj} x^N$.  Compute:
\begin{equation*}
\begin{split}
\sum_k a_{ik} \otimes a_{kj} &= \sum_k \left(\sum_M \chi(M)_{ik}
x^M\right) \otimes \left(\sum_N \chi(N)_{kj} x^N\right) \\
&= \sum_{M,N} \left(\sum_k \chi(M)_{ik} \chi(N)_{kj}\right) x^M
\otimes x^N
\end{split}
\end{equation*}
Realizing that $\sum_k \chi(M)_{ik} \chi(N)_{kj}$ is nothing more
than the $(i,j)^{\text{th}}$ entry of the matrix product $\chi(M)
\chi(N)$, this completes the proof.

\end{proof}

Recalling equations \ref{comodEquation1} and \ref{comodEquation2},
by propositions \ref{UnBigEquationProp} and \ref{UnMiniProp} we
have

\begin{proposition}
The matrix formula $(a_{ij}) = \sum_M \chi(M) x^M$ is a
representation of $U_n$ if and only if
\begin{enumerate}
\item{$\chi(M) = 0$ for all but finitely many $M$}

\item{$\chi(0) = \text{Id}$}

\item{
\begin{equation}
\label{TheMainEquation}
\begin{split}
& \sum_{M,N}  \chi(M) \chi(N) x^M \otimes x^N \\
= \sum_M \chi(M) &\left[ \sum_{S_1 + \ldots + S_n = M} {S_1 +
\ldots + S_n \choose S_1, \ldots, S_n} \left( \prod_{1 \leq i < j
\leq n} x_{ij}^{L_{ij}} \right) \otimes \left( \prod_{1 \leq i < j
\leq n} x_{ij}^{R_{ij}} \right) \right] \\
\end{split}
\end{equation}
}
\end{enumerate}

In particular, the matrix product $\chi(M) \chi(N)$ must be
exactly the coefficient of the monomial tensor $x^M \otimes x^N$
on the right hand side of equation \ref{TheMainEquation}.
\end{proposition}
The left hand side of equation \ref{TheMainEquation} above is a
sum over distinct monomial tensors, and so the coefficient of $x^M
\otimes x^N$ is exactly $\chi(M) \chi(N)$. The right hand side of
this equation, on the other hand, is \emph{not} a sum over
distinct monomial tensors, and for the most part it is a
non-trivial exercise figuring out exactly what the coefficient of
$x^M \otimes x^N$ is for arbitrary $M$ and $N$.  Nonetheless, for
the $M$ and $N$ we are interested in, we shall be able to compute
exactly what the coefficient of the monomial tensor $x^M \otimes
x^N$ is on the right hand side, and in so doing compute $\chi(M)
\chi(N)$, and in so doing prove theorem \ref{TheMainTheoremSub}.

As an example of how we will be using this equation, consider
again the case $n=4$, and consider the matrix product
$\chi(\varepsilon_{14}) \chi(\varepsilon_{13})$. We claim that it
is always equal to $\chi(\varepsilon_{14} + \varepsilon_{13})$. To
show this, consider the equation
\begin{equation*}
\begin{split}
& \quad \quad \quad
x_{12}^{\s{2}{1}{2}+\s{2}{1}{3}+\s{2}{1}{4}}x_{23}^{\s{2}{2}{3}+\s{2}{2}{4}}x_{34}^{\s{2}{3}{4}}
x_{13}^{\s{3}{1}{3} + \s{3}{1}{4}}
x_{24}^{\s{3}{2}{4}}x_{14}^{\s{4}{1}{4}} \\
& \quad \bigotimes \\
&\quad \quad \quad
x_{12}^{\s{1}{1}{2}}x_{23}^{\s{1}{2}{3}+\s{2}{1}{3}}x_{34}^{\s{1}{3}{4}+\s{2}{2}{4}+\s{3}{1}{4}}
x_{13}^{\s{1}{1}{3}} x_{24}^{\s{1}{2}{4}+\s{2}{1}{4}}
x_{14}^{\s{1}{1}{4}}\\
& \quad \quad \quad = x_{14} \otimes x_{13}\\
\end{split}
\end{equation*}
Since the variables $s_{ij}^k$ are demanded to be non-negative
integers, clearly there is only one collection of matrices
$S_1,S_2,S_3,S_4$ in non-negative integer entries that gives a
solution to this equation, namely $\s{4}{1}{4} = 1, \s{1}{1}{3} =
1$, and all other $\s{k}{i}{j} = 0$.  Thus the coefficient of
$x_{14} \otimes x_{13}$ on the right hand side of equation
\ref{TheMainEquation} is exactly
\begin{equation*}
\begin{split}
{S_1+S_2+S_3+S_4 \choose S_1,S_2,S_3,S_4} \chi(S_1 + S_2 +
S_3+S_4) &= {\varepsilon_{13} + 0 + 0 + \varepsilon_{14} \choose
\varepsilon_{13}, 0,0 ,\varepsilon_{14}} \chi(\varepsilon_{13} +
\varepsilon_{14}) \\
&= \chi(\varepsilon_{13} + \varepsilon_{14}) \\
\end{split}
\end{equation*}
The coefficient of $x_{14} \otimes x_{13}$ on the left hand side
of equation \ref{TheMainEquation} is of course
$\chi(\varepsilon_{14}) \chi(\varepsilon_{13})$, whence they are
equal.

\begin{lemma}
\label{LandRlemmaOld}
 Consider the variable expressions
\[ L_{ij} = \sum_{k=j}^n s_{ik}^{j-i+1} \hspace{1cm} R_{ij} =
\sum_{k=1}^i s_{kj}^{i-k+1} \]

\begin{enumerate}
\item{Each of the variables $s_{ij}^k$ occur at most once in any
of the $L_{ij}$, and the only variables that do not occur in any
of the $L_{ij}$ are those of the form $s_{ij}^1$ (i.e.~those
variables occurring in the matrix $S_1$)}

\item{Each of the variables $s_{ij}^k$ occur at most once in any
of the $R_{ij}$, and the only variables that do not occur in any
of the $R_{ij}$ are those of the form $s_{i,k}^{k-i+1}$, for $k>i$
(i.e.~those on the super-diagonal of $S_2$, on the
super-super-diagonal of $S_3$, $\ldots$, on the
$(1,n)^{\text{th}}$ entry of $S_n$)}

\item{The variables that occur in both of the expressions $L_{ij}$
and $R_{uv}$ are exactly
\begin{enumerate}
\item{$s_{i,v}^{j-i+1}$ if $j=u$}

\item{None otherwise}

\end{enumerate}
}
\end{enumerate}

\end{lemma}

\begin{proof}

$(1)$ For the variable $s_{uv}^w$ to occur in $L_{ij}$, we must
have a $k$ with $j \leq k \leq n$ such that $s_{ik}^{j-i+1} =
s_{uv}^w$. This forces $i=u,k=v$, and $w = j-i+1$, i.e.~$j =
w+u-1$.  Thus there is only one possible choice for $i,j$ and $k$.

Now suppose $w=1$.  Then there is no $L_{ij}$ in which $s_{uv}^w$
occurs, since this forces $j=u=i$, and of course $i$ is always
demanded to be less than $j$.  On the other hand, if $2 \leq w
\leq n$, set $i=u, k = v$, and $j = w+u-1$; we claim that
$s_{uv}^w$ occurs as the $k^{\text{th}}$ term of $L_{ij}$.  To
prove this, we need to verify that $1 \leq i < j \leq n$, and that
$j \leq k \leq n$. That $1 \leq i$ is obvious from $1 \leq u$.
 Obviously also $k \leq n$ since $v \leq n$.  Recall that, by assumption
of the form of the variable $s_{uv}^w$, $v-u \leq w -1$, i.e~$u
\leq v-w+1$. Then
\[j = w+u-1 \leq w + (v-w+1)-1 = v \leq n \]
and
\[ j \leq v = k \]
as required.

$(2)$ For the variable $s_{uv}^w$ to occur in $R_{ij}$, we must
have a $k$ with $1 \leq k \leq i$ such that $s_{kj}^{i-k+1} =
s_{uv}^w$. This forces $j=v,k=u$, and $w = i-k+1$, i.e.~$i =
w+u-1$.  Thus there is only one possible choice for $i,j$ and $k$.

Now suppose that $s_{uv}^w$ is of the form $s_{i,k}^{k-i+1}$, for
$k>i$, that is, so that $v-u=w-1$.  Then this forces
\[ j-k = v-u = w-1 =i-k+1-1 = i-k \]
so that $j=i$, which is impossible.  On the other hand, if
$s_{uv}^w$ is not of this form, we can assume that $v-u > w-1$.
Then set $i=w+u-1,j=v,k=u$.  We claim that $s_{uv}^w$ occurs as
the $k^{\text{th}}$ term of $R_{ij}$.  We must verify that $1 \leq
i < j \leq n$, and that $1 \leq k \leq i$.

That $ j \leq n$ is obvious from $v \leq n$, $1 \leq k$ since $1
\leq u$, and $1 \leq i$ since $1 \leq w + u -1$, i.e.~$2 \leq w +
u$.  $i < j$ is equivalent to $w + u -1 < v$, which is equivalent
to $v-u > w-1$, and $k \leq i$ is equivalent to $u \leq w+u-1$,
which is equivalent to $1 \leq w$.

(3) To see which variables occur in both $L_{ij}$ and $R_{uv}$, we
seek $k_0$ and $k_1$ with $j \leq k_0 \leq n$ and $1 \leq k_1 \leq
u$ such that
\[ s_{i,k_0}^{j-i+1} = s_{k_1,v}^{u-k_1+1} \]
which forces $i = k_1$, $v=k_0$, and $j-i+1 = u-k_1 + 1 = u-i+1$,
i.e.~$j=u$.  Thus we must have $j=u$ to have any variables in
common.  If this is the case, then $1 \leq i < j = u < v \leq n$,
whence we have $1 \leq k_1 < j = u < k_0 \leq n$, as required,
giving us the variable $s_{iv}^{j-i+1}$.

\end{proof}

We close this section with one more combinatorial lemma.

\begin{lemma}
\label{BigFatLemma} Let $Y = (y_{ij})$ be an $n \times n$ strictly
upper triangular matrix with zeroes on its top row, and let $Z =
(z_{ij})$ be an $n \times n$ strictly upper triangular matrix with
zeroes in all rows except perhaps the top, both with non-negative
integer entries. Then the system
\[ \left(\prod_{1 \leq i < j \leq n} x_{ij}^{L_{ij}} \right)
\otimes \left(\prod_{1 \leq i < j \leq n} x_{ij}^{R_{ij}} \right)
= x^Y \otimes x^Z
\]
has exactly one solution in the variables $s_{ij}^k$, namely
\begin{enumerate}
\item{$s_{1j}^1 = z_{1j}$ for all $2 \leq j \leq n$}

\item{$s_{ij}^{j-i+1} = y_{ij}$ for all $2 \leq i < j \leq n$}

\item{All other $s_{ij}^k = 0$}
\end{enumerate}

\begin{remark}
In the case of $n=5$, a useful heuristic to remember this solution
is
\[
\left(%
\begin{array}{ccccc}
  0 & s_{12}^1 & s_{13}^1  & s_{14}^1 & s_{15}^1 \\
   & 0 & s_{23}^2 & s_{24}^3 & s_{25}^4 \\
   &  & 0 & s_{34}^2 & s_{35}^3 \\
   &  &  & 0 & s_{45}^2 \\
   &  &  &  & 0 \\
\end{array}%
\right) =
\left(%
\begin{array}{ccccc}
  0 & z_{12} & z_{13}  & z_{14} & z_{15} \\
   & 0 & y_{23} & y_{24} & y_{25} \\
   &  & 0 & y_{34} & y_{35} \\
   &  &  & 0 & y_{45} \\
   &  &  &  & 0 \\
\end{array}%
\right) = Y + Z
\]
with any $s_{ij}^k$ not mentioned assumed to be zero.

\end{remark}

\begin{proof}
 The key fact is that the variables mentioned in
part (1) of the above do not occur in any of the $L_{ij}$, and
those given in part (2) do not occur in any of the $R_{ij}$ (see
lemma \ref{LandRlemmaOld}).

We prove first that the assignments given actually are a solution.
We need to verify that
\[ R_{1,j} = z_{1,j} \text{ for all } 2 \leq j \leq n \]
and
\[ L_{i,j} = y_{i,j} \text{ for all } 2 \leq i < j \leq n \]
with all other $L_{ij},R_{ij} = 0$. Consider
\[ R_{ij} = \sum_{k=1}^i s_{kj}^{i-k+1} = s_{1,j}^i + s_{2,j}^{i-1} \ldots + s_{ij}^1 \]
We see firstly that $R_{1,j} = s_{1,j}^1 = z_{1,j}$.  By lemma
\ref{LandRlemmaOld}, no variable of the form given in part (2) of
the proposition occur in any $R_{ij}$, and if $i>1$, no variable
of the form given in part (1) may occur either.  This gives
$R_{ij} = 0$ for $i>1$, as required.

Consider now
\[ L_{ij} = \sum_{k=j}^n s_{ik}^{j-i+1} = s_{ij}^{j-i+1} +
s_{i,j+1}^{j-i+1} + \ldots + s_{in}^{j-i+1} \]  Note firstly that
no variables of the form given part (1) occur in any $L_{ij}$
(lemma \ref{LandRlemmaOld}). If $i=1$, then we have
\[ L_{1j} = s_{1,j}^{j} +
s_{1,j+1}^{j} + \ldots + s_{1,n}^{j} \] In order for a variable of
the form given in (2) to occur in this expression, we must have
$j=1+j+1-1$, which is absurd; thus $L_{1j} = 0$.  In case $i>1$,
we see there is exactly one term of the form given in (2)
occurring, namely the first term, $s_{i,j}^{j-i+1}$, which is
equal to $y_{ij}$ by fiat.  This gives $L_{ij} = y_{ij}$ for
$i>1$, as required.

We now argue that this is the only possible solution.  First, the
equations $R_{1,j} = s_{1,j}^1 = z_{1,j}$ force the assignments
given in part (1).  Second, if $i \geq 2$, then the variable
$s_{ij}^{j-i+1}$ occurs in $L_{ij}$, and this is the \emph{only}
variable of this form occurring in $L_{ij}$.  Third, if $i \geq
2$, then every other variable occurring in $L_{ij}$, not being of
the form given in (2), occurs in some $R_{uv}$ (lemma
\ref{LandRlemmaOld}), necessarily with $i=u>1$; and since $R_{uv}
= 0$ for all $u>1$, this makes them zero.  Thus, for $i \geq 2$,
$L_{ij} = y_{ij}$ has only one possible non-zero variable, namely
$s_{ij}^{j-i+1}$; this forces $s_{ij}^{j-i+1} = y_{ij}$ for all $i
\geq 2$.

Finally, let $s_{ij}^k$ be any variable not of the form given in
either (1) or (2).  Certainly it occurs somewhere among the
$L_{ij}$ or $R_{ij}$.  If this happens to be any of the non-zero
$L's$ or $R's$, we've already argued that $s_{ij}^k$ must be zero;
and if it is \emph{not} among these $L's$ or $R's$, then that
particular $L$ or $R$ must itself be zero, forcing $s_{ij}^k = 0$.
This completes the proof.
\end{proof}

\end{lemma}

\section{The Main Theorem: Necessity}
\label{necessitySection}

We can now put equation \ref{TheMainEquation} to work in proving
theorem \ref{TheMainTheoremSub}. Recall that $\varepsilon_{ij}$ is
the $n \times n$ matrix with a $1$ in its $(i,j)^{\text{th}}$
entry, zeroes elsewhere. Let $M$ be an arbitrary $n \times n$
strictly upper triangular matrix with non-negative integer
entries.  To begin, we need a factorization for $\chi(M)$ in terms
of those matrices of the form $\chi(r \varepsilon_{ij})$.

\begin{lemma}
\label{UnDecompositionLemma} If $M=(m_{ij})$, then
\[ \chi(M) = \prod_{i=n-1}^1 \prod_{j=i+1}^n \chi(m_{ij}
\varepsilon_{ij}) \]
\end{lemma}

\begin{remark}
The notation ``$\prod_{i=n-1}^1$'' is not a typo; we are listing
the factors in reverse order for convenience.  For example, in the
case of $n=4$, we are saying that
\[ \chi(M) = \chi(m_{34} \varepsilon_{34}) \chi(m_{23} \varepsilon_{23}) \chi(m_{24} \varepsilon_{24})
\chi(m_{12} \varepsilon_{12}) \chi(m_{13} \varepsilon_{13})
\chi(m_{14} \varepsilon_{14})\]
\end{remark}

\begin{proof}
We proceed by induction on $n$.  If $n=2$ the above equation is
$\chi(M) = \chi(m_{12} \varepsilon_{12}) = \chi(M)$, which is
obvious.  Let $Y$ and $Z$ be the $n \times n$ matrices
\[Y =
\left(%
\begin{array}{cccccc}
  0 & 0 & 0 & 0 & \dots & 0 \\
   & 0 & m_{23} & m_{24} & \dots & m_{2n} \\
   &  & 0 & m_{34} & \dots & m_{3n} \\
   &  &  & \ddots & \ddots & \vdots \\
   &  &  &  & 0 & m_{n-1,n} \\
   &  &  &  &  & 0 \\
\end{array}%
\right), \hspace{.2cm} Z =
\left(%
\begin{array}{cccccc}
  0 & m_{12} & m_{13} & m_{14} & \dots & m_{1n} \\
   & 0 & 0 & 0 & \dots & 0 \\
   &  & 0 & 0 & \dots & 0 \\
   &  &  & \ddots & \ddots & \vdots \\
   &  &  &  & 0 & 0 \\
   &  &  &  &  & 0 \\
\end{array}%
\right)
\]
i.e~the $n \times n$ matrices with the top row of $M$ deleted, and
\emph{all but} the top row deleted. Using the second embedding of
$U_{n-1}$ into $U_n$ described on page \pageref{embeddings}, we
conclude by induction that
\begin{equation*}
\begin{split}
\chi(Y) &= \chi(m_{n-1,n} \varepsilon_{n-1,n}) \chi(m_{n-2,n-1}
\varepsilon_{n-2,n-1}) \chi(m_{n-2,n} \varepsilon_{n-2,n}) \ldots
\chi(m_{23} \varepsilon_{23}) \ldots \chi(m_{2n} \varepsilon_{2n})
\\
& = \prod_{i=n-1}^2 \prod_{j=i+1}^n \chi(m_{ij} \varepsilon_{ij})
\end{split}
\end{equation*}
And by using the first embedding $G_a^{n-1}$ into $U_n$  described
on page \pageref{embeddings}, we have
\begin{equation*}
\begin{split}
 \chi(Z) &= \chi(m_{12} \varepsilon_{12}) \ldots \chi(m_{1n}
\varepsilon_{1n}) \\
 &=\prod_{j=2}^n \chi(m_{1j} \varepsilon_{1j})
\end{split}
\end{equation*}
(See theorem 12.3.6 of \cite{MyDissertation} for an account of the
representation theory of direct products of the Additive group.)

Now consider the matrix product $\chi(Y) \chi(Z)$; we claim it is
equal to $\chi(Y+Z) = \chi(M)$.  Using equation
\ref{TheMainEquation}, we seek solutions to the system
\[ \left(\prod_{1 \leq i < j \leq n} x_{ij}^{L_{ij}} \right)
\otimes \left(\prod_{1 \leq i < j \leq n} x_{ij}^{R_{ij}} \right)
= x^Y \otimes x^Z
\]
By lemma \ref{BigFatLemma} this system has exactly one solution,
and we leave it to the reader that this solution gives
\[{S_1 + \ldots + S_n \choose S_1, \ldots, S_n} = 1\]
and
\[ S_1 + \ldots + S_n = M \]
This gives
\begin{equation*}
\begin{split}
 \chi(M) &= \chi(Y) \chi(Z) \\
 &= \left(\prod_{i=n-1}^2 \prod_{j=i+1}^n \chi(m_{ij}
 \varepsilon_{ij}\right)\left(\prod_{j=2}^n \chi(m_{1j} \varepsilon_{1j}
 \right)\\
 &= \prod_{i=n-1}^1 \prod_{j=i+1}^n \chi(m_{ij}
\varepsilon_{ij})
\end{split}
\end{equation*}
This completes the proof.

\end{proof}

For any $1 \leq i < j \leq n$, those matrices of the form $1 +
x_{ij}\varepsilon_{ij}$ form a subgroup of $U_n$ isomorphic to the
Additive group $G_a$, and we shall use this fact repeatedly.  We
shall need the following concerning the representation theory of
the Additive group $G_a$ over prime characteristic fields.  In
what follows we shall identify $G_a$ as $U_2$, i.e.~the collection
of all matrices of the form
\[
\left(%
\begin{array}{cc}
  1 & x \\
  0 & 1 \\
\end{array}%
\right)
\]

The following proposition is originally due to A Suslin, E M
Friedlander and C P Bendel, 1997, and can be found in the proof of
proposition 1.2 of \cite{SFB}.  However, our notation differs
markedly from theirs, and so the reader may with to consult
section 12.3 of \cite{MyDissertation} instead.

\begin{proposition}
\label{GaProp} Let $k$ be a field of characteristic $p>0$, and let
$(a_{ij}) = \sum_{n \in \mathbb{N}} \chi(n \varepsilon_{12}) x^n$
be any representation of $G_a$ over $k$. Set $X_0 =
\chi(\varepsilon_{12}), X_1 = \chi(p\varepsilon_{12}), \ldots, X_m
= \chi(p^m \varepsilon_{12})$, with $X_i = 0$ for $i > m$.
\begin{enumerate}
\begin{item}
$(a_{ij})$ can be written as
\[ e^{xX_0}e^{x^pX_1} \ldots e^{x^{p^m} X_m} \]
The $X_i$ all commute and are nilpotent of order no greater than
$p$. Further, any finite collection of such matrices defines a
representation according to this formula.
\end{item}
\begin{item}
Let $r \in \mathbb{N}$, and write $r = r_0 + r_1 p + \ldots + r_m
p^m$ in $p$-ary notation.  Then
\[ \chi(r \varepsilon_{12}) = \Gamma(r)^{-1}
\chi(\varepsilon_{12})^{r_0} \chi(p \varepsilon_{12})^{r_1} \ldots
\chi(p^m \varepsilon_{12})^{r_m} \] where $\Gamma(r)
\stackrel{\text{\emph{def}}}{=} r_0!r_1! \ldots r_m!$.

\end{item}

\end{enumerate}

\end{proposition}

\begin{lemma}
\label{UntoGaLemma} Let $k$ be a field of characteristic $p>0$,
$(a_{ij}) = \sum_M \chi(M) x^M$ a representation of $U_n$ over
$k$.  For $r \in \mathbb{N}$, write $r = r_0 + r_1 p + \ldots +
r_m p^m$ in $p$-ary notation. Then for any $1 \leq i < j \leq n$
and $r \in \mathbb{N}$,
\[ \chi(r \varepsilon_{ij}) = \Gamma(r)^{-1}
\chi(\varepsilon_{ij})^{r_0} \chi(p \varepsilon_{ij})^{r_1} \ldots
\chi(p^m \varepsilon_{ij})^{r_m} \] where $\Gamma(r)
\stackrel{\text{def}}{=} r_0! r_1! \ldots r_m!$.  Further, all of
the factors of this product commute, and all of the $\chi(p^m
\varepsilon_{ij})$ are nilpotent of order no greater than $p$.
\end{lemma}

\begin{proof}
For fixed $i$ and $j$, the collection of those matrices of $U_n$
the form $1+x_{ij} \varepsilon_{ij}$ form a subgroup of $U_n$
isomorphic to the additive group $G_a$.  This embedding is given
by the Hopf algebra map $A_n \rightarrow k[x]$ which sends
$x_{ij}$ to $x$ and all other $x_{rs}$ to zero.  Apply proposition
\ref{GaProp}.

\end{proof}

The next lemma, while simple, is the crucial fact which allows us
to prove the main theorem of this paper.  It also illustrates why
we do \emph{not} suspect a result analogous to our main theorem to
hold for non-unipotent algebraic groups (i.e.~groups with a
non-nilpotent Lie algebra).

\begin{lemma}
\label{carryingLemma} Let $k$ be a field of characteristic $p>0$,
$(a_{ij}) = \sum_M \chi(M) x^M$ a $d$-dimensional representation
of $U_n$ over $k$, and suppose that $p \geq 2d$. Let $r,s \in
\mathbb{N}$, and suppose that the sum $r + s$ carries modulo $p$.
Then for any $i,j,u,v$, at least one of $\chi(r \varepsilon_{ij})$
or $\chi(s \varepsilon_{uv})$ must be zero.
\end{lemma}

\begin{proof}
Write $r = r_0 + r_1 p + \ldots + r_m p^m$ in p-ary notation,
similarly for $s$, and since $r+s$ carries, let $r_t + s_t \geq
p$.  Then at least one of $r_t$ or $s_t$ is greater than or equal
to $p/2$, say $r_t$.  By lemma \ref{UntoGaLemma} write
\[ \chi(r \varepsilon_{ij}) = \Gamma(r)^{-1}
\chi(\varepsilon_{ij})^{r_0} \chi(p \varepsilon_{ij})^{r_1} \ldots
\chi(p^t \varepsilon_{ij})^{r_t} \ldots \chi(p^m
\varepsilon_{ij})^{r_m}
\]
Then $\chi(r \varepsilon_{ij})$ is zero, since $\chi(p^t
\varepsilon_{ij})^{r_t}$ is zero, since $\chi(p^t
\varepsilon_{ij})$ is nilpotent of order no greater than $d \leq
p/2 \leq r_t$.

\end{proof}

We are now ready to prove theorem \ref{TheMainTheoremSub}. As
stated in the introduction on page \pageref{embeddings}, the four
embeddings mentioned there give us a great deal for free; all that
needs to be checked is the bracket $[\chi(p^m \varepsilon_{ij}),
\chi(p^n \varepsilon_{tu})]$ when $(i,j)$ is in the top row and
$(t,u)$ in the right-most column, or when $(i,j) = (1,n)$ and
$(t,u)$ is in neither the top row nor the right-most column. We
break it down into cases.

Recall that the Lie algebra of $U_n$, which we identify as
$\varepsilon_{ij}, 1 \leq i < j \leq n$, has the following Lie
bracket:
\[ [\varepsilon_{rs},\varepsilon_{tu}] = \left\{
\begin{array}{cc}
  \varepsilon_{ru} & \hspace{.3cm}\text{ if } s = t\\
  -\varepsilon_{ts} & \hspace{.3cm}\text{ if } r=u \\
  0 & \hspace{.3cm}\text{  otherwise
  }\\
\end{array}
\right.
\]

For the rest of this section we assume that $k$ is a field of
characteristic $p$, that $p \geq 2d$, and that the $\chi(M)$
correspond to a $d$-dimensional representation of $U_n$ over $k$.

\begin{proposition}
\label{UnProp1} For any $l,m,t$ and $u$,
\[ \chi(p^l \varepsilon_{tn}) \chi(p^m \varepsilon_{1u}) = {p^l \varepsilon_{tn} + p^m \varepsilon_{1u}
 \choose p^l \varepsilon_{tn}, p^m \varepsilon_{1u}  }\chi(p^l
\varepsilon_{tn}+p^m \varepsilon_{1u}) \]

\end{proposition}

\begin{proof}
We seek solutions to $L_{tn} = s_{tn}^{n-t+1} = p^l, R_{1u} =
s_{1u}^1 = p^m$, and all other $L_{ij},R_{ij} = 0$.   Notice
firstly that the variable $s_{tn}^{n-t+1}$ does not occur in any
of the $R_{ij}$, and that $s_{1u}^1$ does not occur in any of the
$L_{ij}$, and that neither occur in any more $L's$ or $R's$ (lemma
\ref{LandRlemmaOld}), so there is no conflict in setting $L_{tn} =
p^l, R_{1u} = p^m$; that is, we have at least one solution.  But
also by lemma \ref{LandRlemmaOld}, every other variable occurs in
at least one of the other $L_{ij}$ or $R_{ij}$, forcing them to be
zero, giving us exactly one solution, namely $s_{tn}^{n-t+1} =
p^l$, $s_{1u}^1 = p^m$, and all other $s_{ij}^k = 0$.  This gives
us that the coefficient of the monomial tensor $x_{tn}^{p^l}
\otimes x_{1u}^{p^m}$ is exactly
\[ {p^l \varepsilon_{tn} + p^m \varepsilon_{1u} \choose p^l
\varepsilon_{tn}, p^m \varepsilon_{1u} } \chi(p^l \varepsilon_{tn}
+ p^m \varepsilon_{1u}) \] and the proposition is proved.

\end{proof}

\begin{proposition}
\label{UnProp2} If $t \neq u$, then
\[ \chi(p^l \varepsilon_{1u}) \chi(p^m \varepsilon_{tn}) = {p^l
\varepsilon_{1u}+ p^m \varepsilon_{tn} \choose  p^m
\varepsilon_{tn},p^m \varepsilon_{tn}} \chi(p^l \varepsilon_{1u}+
p^m \varepsilon_{tn}) \]
\end{proposition}

\begin{proof}
We again examine
\[L_{1u} = \sum_{k=u}^n s_{1,k}^u \quad \text{ and } \quad  R_{tn} = \sum_{k=1}^t s_{t-k+1,n}^k \]
We seek solutions to $L_{1u} = p^l, R_{tn} = p^m$, and all other
$L_{ij},R_{ij} = 0$.  We claim there is exactly one solution,
namely $s_{1u}^u = p^l$, $s_{tn}^1 = p^m$, and all other $s_{ij}^k
= 0$.  First, $L_{1u}$ and $R_{tn}$ share no variables in common.
Second, the only variable occurring in $L_{1u}$ which
\emph{doesn't} also occur in some $R_{ij}$ is $s_{1u}^u$, and the
only variable occurring in $R_{tn}$ which \emph{doesn't} occur in
some $L_{ij}$ is $s_{tn}^1$ (lemma \ref{LandRlemmaOld}). This
forces all of the variables occurring in $L_{1u}$ or $R_{tn}$ to
be zero, except for these two, forcing $s_{1u}^u = p^l$, $s_{tn}^1
= p^m$. This gives that the coefficient of the monomial tensor
$x_{1u}^{p^l} \otimes x_{t4}^{p^m}$ in equation
\ref{TheMainEquation} is exactly
\[ {p^l
\varepsilon_{1u}+ p^m \varepsilon_{t4} \choose  p^m
\varepsilon_{t4},p^m \varepsilon_{t4}} \chi(p^l \varepsilon_{1u}+
p^m \varepsilon_{t4}) \] which proves the proposition.
\end{proof}

\begin{corollary}
If $t \neq u$, then for any $l$ and $m$,
\[ [\chi(p^l \varepsilon_{1u}), \chi(p^m \varepsilon_{tn})] = 0 \]
\end{corollary}

\begin{proof}
By propositions \ref{UnProp1} and \ref{UnProp2}, $\chi(p^l
\varepsilon_{1u})\chi(p^m \varepsilon_{tn}) = \chi(p^m
\varepsilon_{tn})\chi(p^l \varepsilon_{1u})$.

\end{proof}

\begin{proposition}
For any $1 \leq t < u \leq n$ and any $l$ and $m$,
\[ [\chi(p^l \varepsilon_{1n}), \chi(p^m \varepsilon_{tu})] = 0 \]
\end{proposition}

\begin{proof}
Note that $L_{1n} = s_{1n}^n$ and that $R_{1n} = s_{1n}^1$, both
consisting of a single variable. We leave it to the reader then to
compute the products $\chi(p^l \varepsilon_{1n})\chi(p^m
\varepsilon_{tu})$ and $ \chi(p^m \varepsilon_{tu})\chi(p^l
\varepsilon_{1n})$, to see that each have exactly one solution,
and these solutions give the same answer for both.
\end{proof}

Thus far we have not used at all the assumption that $p \geq 2d$.
This will change as we consider the last case, namely $[\chi(p^l
\varepsilon_{1u}), \chi(p^m \varepsilon_{un})]$.

\begin{proposition}
For any $l,m$, and $1 < u \leq n$,

\[ [\chi(p^l \varepsilon_{1u}), \chi(p^m \varepsilon_{un})] =
\sum_{k=1}^{\text{min}(p^l,p^m)} \chi((p^m -k)
\varepsilon_{un})\chi((p^l-k)\varepsilon_{1,u}) \chi(k
\varepsilon_{1n}) \]

\end{proposition}

\begin{proof}
We shall compute the product $\chi(p^l \varepsilon_{1u}) \chi(p^m
\varepsilon_{un})$.  Thus we seek solutions to
\[ L_{1u} = p^l \text{ and } R_{un} = p^m \]
with all other $L_{ij},R_{ij}$ equal to zero.  Firstly, $L_{1u}$
and $R_{un}$ have exactly one variable in common, namely
$s_{1n}^u$. Secondly, there is exactly one variable occurring in
$L_{1u}$ which doesn't occur in any $R_{ij}$, namely $s_{1u}^u$,
and exactly one variable occurring in $R_{un}$ which doesn't occur
in any of the $L_{ij}$, namely $s_{un}^1$ (lemma
\ref{LandRlemmaOld}). This gives
\[ s_{1u}^u + s_{1n}^u = p^l \text{ and } s_{un}^1 + s_{1n}^u
= p^m \] with all other $s_{ij}^k$ equal to zero.  Clearly then,
every non-negative integer value of $s_{1n}^u$ no greater than
either $p^l$ or $p^m$ gives a solution, and these are the only
solutions.  If $k = s_{1n}^u$ is any such value, its contribution
to the coefficient of $x_{1u}^{p^l} \otimes x_{un}^{p^m}$ in
equation \ref{TheMainEquation} is
\[ {k \varepsilon_{1n} + (p^l-k)\varepsilon_{1u} + (p^m -k)
\varepsilon_{un} \choose k
\varepsilon_{1n},(p^l-k)\varepsilon_{,u},(p^m -k)
\varepsilon_{un}} \chi(k \varepsilon_{1n} +
(p^l-k)\varepsilon_{1u} + (p^m -k) \varepsilon_{un}) \] Note that
it is impossible for any of the tuples $(1,n),(n,u)$ or $(u,n)$ to
be equal, so the above multinomial coefficient is exactly $1$.
Thus we can write
\[ \chi(p^l \varepsilon_{1u}) \chi(p^m
\varepsilon_{un}) = \sum_{k=0}^{\text{min}(p^l,p^m)}\chi(k
\varepsilon_{1n} + (p^l-k)\varepsilon_{1u} + (p^m -k)
\varepsilon_{un}) \] and by lemma \ref{UnDecompositionLemma} we
can further write
\[ \chi(p^l \varepsilon_{1u}) \chi(p^m
\varepsilon_{un}) = \sum_{k=0}^{\text{min}(p^l,p^m)} \chi((p^m -k)
\varepsilon_{un})\chi((p^l-k)\varepsilon_{1u}) \chi(k
\varepsilon_{1n}) \] Finally, note that the value $k=0$ gives
$\chi(p^m \varepsilon_{un}) \chi(p^l \varepsilon_{1u})$, so we can
write
\[ \chi(p^l \varepsilon_{1u}) \chi(p^m
\varepsilon_{un}) - \chi(p^m \varepsilon_{un}) \chi(p^l
\varepsilon_{1u}) = \sum_{k=1}^{\text{min}(p^l,p^m)} \chi((p^m -k)
\varepsilon_{un})\chi((p^l-k)\varepsilon_{1u}) \chi(k
\varepsilon_{1n}) \] which proves the proposition.

\end{proof}

\begin{corollary}
If $p \geq 2d$, then
\[ [\chi(p^l \varepsilon_{1u}), \chi(p^m \varepsilon_{un})] = \left\{
\begin{array}{cc}
  0 & \hspace{.3cm}\text{ \emph{if }} l \neq m \\
  \chi(p^m \varepsilon_{1n}) & \hspace{.3cm}\text{ otherwise
  }\\
\end{array}
\right.
\]

\end{corollary}

\begin{proof}
By the previous proposition, examine \[[\chi(p^l
\varepsilon_{1u}), \chi(p^m \varepsilon_{un})] =
\sum_{k=1}^{\text{min}(p^l,p^m)} \chi((p^m -k)
\varepsilon_{un})\chi((p^l-k)\varepsilon_{1,u}) \chi(k
\varepsilon_{14}) \] First suppose that $l \neq m$, say that $l <
m$.  Then for every value $k$ of this summation, there is clearly
some `carrying' in computing the sum $(p^m - k) + k$. Then by
lemma \ref{carryingLemma}, at least one of $\chi((p^m-k)
\varepsilon_{un})$ or $\chi(k \varepsilon_{1n})$ is always zero,
forcing every term in the summation to be zero.  This gives
$[\chi(p^l \varepsilon_{1u}), \chi(p^m \varepsilon_{un})] = 0$; an
analogous proof holds in case $l > m$.

Now suppose that $l=m$.  Then for the same reason every term in
the above summation is zero, \emph{except} for the last term
$k=p^m$, since $(p^m - p^m) + p^m$ does not carry.  This gives
\[[\chi(p^m \varepsilon_{1u}),
\chi(p^m \varepsilon_{un})] = \chi(0 \varepsilon_{un}) \chi(0
\varepsilon_{1u}) \chi(p^m \varepsilon_{1n})\] Recalling that
$\chi(0) = \text{Id}$, this completes the proof.
\end{proof}

Theorem \ref{TheMainTheoremSub} is now proved, and hence also part
(2) of our main theorem, \ref{TheMainTheorem}.

\section{The Baker-Campbell-Hausdorff formula}

The proof of part (1) of our main theorem (\ref{TheMainTheorem})
will require some understanding of the Baker-Campbell-Hausdorff
formula.  For the reader's convenience, here we briefly review the
calculation of said formula (here on referred to as $BCH$) as
computed by E.~B.~Dynkin for the case of characteristic zero in
\cite{dynkin1}, \cite{dynkin2} and \cite{dynkin3}. The principal
result of course is that, for not-necessarily-commutative
variables $x$ and $y$ in an associative algebra over a
characteristic zero field, the series $\log(e^x e^y)$ can be
expressed as a formal infinite series of brackets of $x$ and $y$,
brackets of brackets of $x$ and $y$, etc., all with coefficients
given by rational numbers.

The details are important to us; we will claim later that these
arguments apply just as well to the characteristic $p>0$ setting,
under some additional assumptions that disallow the appearance of
any rational numbers with denominators divisible by $p$.

In what follows we shall follow Dynkin in using the less
cumbersome notation $x \circ y = xy-yx$ for the commutator
operator instead of $[x,y]$.  As $\circ$ is by no means
associative, when we write $x_1 \circ x_2 \circ \ldots \circ x_n$,
we shall take it to be \emph{left}-nested; e.g., $x_1 \circ x_2
\circ x_3 \circ x_4 \stackrel{\text{def}}{=} (((x_1 \circ x_2)
\circ x_3) \circ x_4)$.

\begin{lemma}
\label{theLogSeriesLemma} Define the formal infinite series
\[ e^x = \sum_{k=0}^\infty \frac{x^k}{k!} \qquad \log(x) =
\sum_{k=1}^\infty \frac{(-1)^{k-1}}{k} (x-1)^k \] Then $\log(e^x
e^y)$ can be written

\[ \sum \frac{(-1)^{k-1}}{k} \frac{1}{p_1!q_1!\ldots p_k!q_k!}x^{p_1}
y^{q_1} \ldots x^{p_k} y^{q_k} \] where the summation runs over
all tuples of non-negative integers $(p_1, q_1, \ldots, p_k, q_k)$
with the property that $p_i + q_i \neq 0$ (the length $2k$ of the
tuples vary arbitrarily).
\end{lemma}

\begin{proof}

The trick is in collecting the terms correctly.  Write
\[ e^x e^y = \sum_{r=0}^\infty \sum_{s=0}^r \frac{x^s
y^{r-s}}{s!(r-s)!} \] and hence
\[ \log(e^x e^y) = \sum_{k=1}^\infty \frac{(-1)^{k-1}}{k}
\left(\sum_{r=1}^\infty \sum_{s=0}^r \frac{x^s y^{r-s}}{s!(r-s)!}
\right)^k \] An induction argument shows that, for fixed $k$,
\[\left(\sum_{r=1}^\infty \sum_{s=0}^r \frac{x^s
y^{r-s}}{s!(r-s)!} \right)^k = \sum_{{(p_1,q_1, \ldots, p_k,q_k)}
\atop {p_i + q_i \neq 0}} \frac{1}{p_1!q_1!\ldots p_k!q_k!}x^{p_1}
y^{q_1} \ldots x^{p_k} y^{q_k}  \] whence summing over all $k$ we
get
\[ \sum \frac{(-1)^{k-1}}{k} \frac{1}{p_1!q_1!\ldots p_k!q_k!}x^{p_1}
y^{q_1} \ldots x^{p_k} y^{q_k} \]

\end{proof}

From this expression we write
\begin{equation} \label{homogenousBCHformula}
\log(e^x e^y) = \sum_{m=1}^\infty P_m(x,y)
\end{equation}
where $P_m(x,y)$ is a homogeneous polynomial of degree $m$; for
example, $P_1(x,y) = x+y$, $P_2(x,y) = \frac{1}{2}(xy-yx)$, and
$P_3(x,y) =
\frac{1}{12}x^2y-\frac{1}{6}xyx+\frac{1}{12}xy^2+\frac{1}{12}y^2x+\frac{1}{12}yx^2-\frac{1}{6}yxy$.

Let $k$ be a field of characteristic zero and let $R$ be the free
associative (and non-commutative) algebra over $k$ on the
generators $x$ and $y$.  Define a linear mapping $\phi$ from $R$
to itself as follows: it sends the monomial $x_1x_2\ldots x_n$ to
\[ \frac{1}{n} x_1 \circ x_2 \circ \ldots \circ x_n \]
(here the $x_i$ can be either of $x$ or $y$).  For example,
$\phi(x^2y + x + 2xyx) = \frac{1}{3}x \circ x \circ y + x +
\frac{2}{3} x \circ y \circ x$.

\begin{proposition}
\label{allIsLeftNestedProp} Any bracket expression can be written
as a linear combination of left-nested bracket expressions, all of
length no greater than the original, and with no new coefficients
up to perhaps a negation.
\end{proposition}

\begin{proof}
Clearly it suffices to prove that, if $P$ and $Q$ themselves are
left-nested brackets, then $P \circ Q$ can be written as a linear
combination of left-nested brackets. Let $\phi(n)$ be the
statement ``if $P$ is a left-nested bracket of any length, and if
$Q$ is a left-nested bracket of length $\leq n$, then $P \circ Q$
can be written as a linear combination of left-nested brackets''.
Certainly $\phi(1)$ is true.  Now suppose that $\phi(n)$ is true,
and let $P$ be any left-nested bracket, $Q \circ x$ a left-nested
bracket expression of length $n+1$, so necessarily $Q$ is
left-nested of length $n$. Then (by Jacobi and anti-commutativity)
\begin{equation*}
\begin{split}
 P \circ (Q \circ x) &= -Q \circ (x \circ P) -x \circ (P \circ Q)
 \\
 &= (x \circ P) \circ Q + (P \circ Q) \circ x \\
 &= -(P \circ x) \circ Q + (P \circ Q) \circ x \\
\end{split}
\end{equation*}

$P \circ x$ is left-nested, and since $\phi(n)$ is true, $-(P
\circ x) \circ Q$ is a linear combination of left-nested brackets.
The same can be said of $P \circ Q$, whence (by linearity of
$\circ$) $(P \circ Q) \circ x$ is a linear combination of
left-nested brackets, whence $P \circ (Q \circ x)$ is also a
linear combination of left-nested brackets.
\end{proof}

\begin{proposition}
\label{phiIsBeautifulProp} If $x_1 \circ x_2 \circ \ldots \circ
x_n$ is any left-nested bracket expression, then $\phi(x_1 \circ
x_2 \circ \ldots \circ x_n) = x_1 \circ x_2 \circ \ldots \circ
x_n$.
\end{proposition}

\begin{proof}
See pages 31-32 of \cite{dynkin2}.
\end{proof}

The proposition predicts, for example, that $\phi(x \circ y \circ
z) = \phi(xyz - yxz - zxy + zyx) = \frac{1}{3}(x \circ y \circ z)
- \frac{1}{3} y \circ x \circ z - \frac{1}{3} z \circ x \circ y +
\frac{1}{3} z \circ y \circ x) = x \circ y \circ z$.

\begin{proposition}
The homogeneous polynomials $P_i(x,y)$ in formula
\ref{homogenousBCHformula} can all be written as linear
combinations (using only rational numbers) of nested commutators
of $x$ and $y$.
\end{proposition}

\begin{proof}
See \cite{dynkin3}.
\end{proof}

This last proposition, along with propositions
\ref{allIsLeftNestedProp} and \ref{phiIsBeautifulProp} give

\begin{proposition}
If $P_i(x,y)$ is any of the homogeneous polynomials in formula
\ref{homogenousBCHformula}, then $\phi(P_i(x,y)) = P_i(x,y)$.
\end{proposition}

This proposition gives us not only the assurance that each
$P_i(x,y)$ can be written as a rational linear combination of
bracket expressions in $x$ and $y$, but gives us an explicit
method for doing so.  For example, $P_2(x,y) = \frac{1}{2}xy -
\frac{1}{2} yx$, and if we were so dull as to not realize it, we
merely apply $\phi$ to yield
\[ P_2(x,y) = \phi(P_2(x,y)) = \frac{1}{4} x \circ y - \frac{1}{4} y \circ x
= \frac{1}{4} x \circ y + \frac{1}{4} x \circ y = \frac{1}{2} x
\circ y\]

With a view towards proving part (1) of theorem
\ref{TheMainTheorem}, we make the following simple observations
concerning all of this.

\begin{proposition}
\label{BCHinCharpProp} Let $p$ be a prime.

\begin{enumerate}

\item{If $m < p$, then $P_m(x,y)$ contains no coefficients whose
denominators are divisible by $p$.}

\item{If $m < p$, then $\phi(P_m(x,y))$ also contains no
coefficients whose denominators are divisible by $p$.}

\item{Let $X$ and $Y$ be members of a nilpotent matrix Lie algebra
over a field $k$ of characteristic $p$, of nilpotent order no
greater than $p$, and suppose that $X$ and $Y$ themselves are
nilpotent of order no greater than $p$. Then
\[ \log(e^X e^Y) = \sum_{i=1}^{p-1} P_i(x,y) \]}
\end{enumerate}

\end{proposition}

\begin{proof}
From the description of the series for $\log(e^x e^y)$ given in
\ref{theLogSeriesLemma} we glean
\[ P_m(x,y) = \sum \frac{(-1)^{k-1}}{k} \frac{1}{p_1!q_1!\ldots p_k!q_k!} x^{p_1}
y^{q_1} \ldots x^{p_k} y^{q_k} \] where the summation is take over
all tuples of non-negative integers, of whatever length, with the
property that $p_i + q_i > 0$, and that $\sum_{i,j} p_i + q_i =
m$.  As $m < p$, clearly we can have no tuple with some $p_i$ or
$q_i$ divisible by $p$, and neither can we have any $k \geq p$,
and claim (1) is established.

For claim (2), simply realize that the only difference in the
coefficients for $P_m(x,y)$ and the equivalent expression
$\phi(P_m(x,y))$ is a multiplication by a factor of $\frac{1}{m}$
for each coefficient, which again cannot contribute any
denominators divisible by $p$.

For (3), since $X$ and $Y$ are nilpotent of order $\leq p$, both
the series $e^X$ and $e^Y$ are well-defined.  Further, since $X$
and $Y$ belong to a nilpotent Lie algebra of order $\leq p$, any
bracket expressions among $X$ and $Y$ of length $p$ or more
vanish; accordingly, as $P_m(X,Y) = \phi(P_m(X,Y))$ consists
solely of bracket expressions of length $m$, we conclude that
$P_m(X,Y)$ vanishes for all $m \geq p$, whence $\log(e^X e^Y) =
\sum_{i=1}^{p-1} P_i(x,y)$.

\end{proof}


\section{The Main Theorem: Sufficiency}

\label{sufficiencySection}

In this section we prove part (1) of theorem \ref{TheMainTheorem}.

To start, we would like to know when a given Lie algebra
homomorphism $\phi:\liefont{u}_n \rightarrow \liefont{gl}_d$, over
whatever field, can be lifted to a representation of $U_n$.  For
this, we ask how much of the Baker-Campbell-Hausdorff formula do
we need to actually be true in this setting. The following is an
obvious adaptation of analogous facts for characteristic zero
fields.

\begin{proposition}
\label{BCHcaresProp} Let $k$ be any field, and let
$\phi:\liefont{u}_n \rightarrow \liefont{gl}_d$ be a Lie algebra
homomorphism. Suppose that
\begin{enumerate}
\item{The series $e^{\phi(X)}$ ``makes sense'' for all $X \in
\liefont{u}_n$}

\item{The series $\log(g)$ ``makes sense'' for all $g \in U_n$,
and $\log(g)$ is a member of $\liefont{u}_n$}

\item{For all $X,Y \in \liefont{u}_n$, $\log(e^X e^Y)$ exists,
denoted as $\text{BCH}(X,Y)$.  Further, $\text{BCH}(X,Y)$ can be
written \underline{uniformly} (the same for all $X$ and $Y$ in
$\liefont{u}_n$) as a \underline{finite} linear combination of
brackets of $X$ and $Y$, brackets of brackets of $X$ and $Y$,
etc.}

\item{For all $X,Y \in \liefont{u}_n$,
$\log(e^{\phi(X)}e^{\phi(Y)})$ exists, and can be written
\underline{uniformly} as $\text{BCH}(\phi(X),\phi(Y))$ as in (3)}

\end{enumerate}

Then the formula $\Phi(g) \stackrel{\text{\emph{defn}}}{=}
e^{\phi(\log(g))}$ defines a $d$-dimensional representation of
$U_n$ over $k$.

\end{proposition}

\begin{proof}
Let $g,h \in U_n$, and by (2) let $X = \log(g)$, $Y = \log(h)$.
Then
\begin{equation*}
\begin{split}
\Phi(gh) &= \Phi(e^X e^Y) \\
&= \Phi(e^{\text{BCH}(X,Y)}) \quad \text{ by (3) } \\
&= e^{\phi(\text{BCH}(X,Y))} \quad \text{ by definition of $\Phi$
}
\\
&=e^{\text{BCH}(\phi(X),\phi(Y))} \quad \text{ because $\phi$
preserves
brackets } \\
&=e^{\phi(X)} e^{\phi(Y)} \quad \text{ by (4) } \\
&= \Phi(e^X) \Phi(e^Y) \quad \text{ by definition of $\Phi$ } \\
&=\Phi(g)\Phi(h) \\
\end{split}
\end{equation*}

\end{proof}

\begin{theorem}
\label{theCheckListTheorem} Let $k$ be a field of characteristic
$p>0$, and suppose $p \geq \text{max}(n,d)$.  Let
$\phi:\liefont{u}_n \rightarrow \liefont{gl}_d$ be a Lie algebra
homomorphism such that $\phi(X)$ is a nilpotent matrix for each $X
\in \liefont{u}_n$.  Then the formula
\[ \Phi(g) = e^{\phi(\log(g))} \]
defines a $d$-dimensional representation of $U_n$.

\end{theorem}

\begin{proof}
We shall go through the checklist of proposition
\ref{BCHcaresProp} to see that they are all satisfied.

(1): The expression $e^{\phi(X)}$ makes sense since $\phi(X)$ is
nilpotent, of order no greater than $d \leq p$ (that is, the
series expansion for $e^{\phi(X)}$ terminates before getting to
see denominators divisible by $p$).

(2): Every $g \in U_n$ is unipotent, whence $g-1$ is nilpotent of
order no greater than $n \leq p$, whence the series $\log(g) =
\sum_{k=1}^{n-1} \frac{(-1)^{k-1}}{k} (g-1)^k$ likewise terminates
before denominators divisible by $p$ occur.

(3): $\liefont{u}_n$ is a nilpotent Lie algebra of order $n$, and
each $X \in \liefont{u}_n$ is itself nilpotent.  Apply part 3.~of
proposition \ref{BCHinCharpProp}.

(4): As $\phi$ is a Lie algebra homomorphism, its image is also a
nilpotent Lie algebra, of nilpotent order no greater than $n \leq
p$.  Again apply part 3.~of proposition \ref{BCHinCharpProp}.

\end{proof}

Recall that, if $M(x_1, \ldots, x_n)$ is a matrix with entries
taken from the algebra $k[x_1, \ldots, x_n]$ in the commuting
indeterminates $x_1, \ldots, x_n$, then $M^{[m]}$ denotes the
matrix $M(x_1^m, \ldots, x_n^m)$.

\begin{lemma}
\label{pDoesItsThingLemma}

Suppose $(a_{ij})$ is the matrix formula for a representation of
$U_n$ over a field $k$ of characteristic $p$, with comodule map $V
\stackrel{\rho}{\longrightarrow} V \otimes A_n$.  Then
$(a_{ij})^{[p]}$ is also a representation of $U_n$, with comodule
map given by he composition $V \stackrel{\rho}{\longrightarrow} V
\otimes A_n \stackrel{1 \otimes [p]}{\longrightarrow} A_n$, where
$A_n \stackrel{[p]}{\longrightarrow} A_n$ is the linear map which
carries each monomial to its $p^{th}$ power.

\end{lemma}

\begin{proof}

Consider
\begin{diagram}
V & \rTo^\rho & V \otimes A_n & \rTo^{1 \otimes [p]} & V \otimes A \\
\dTo^\rho & & \dTo_{1 \otimes \Delta} & & \dTo_{1 \otimes \Delta} \\
V \otimes A_n & \rTo_{\rho \otimes 1} & V \otimes A_n \otimes A_n  & \rTo^{1 \otimes [p] \otimes [p]} & V \otimes A_n \otimes A_n\\
\dTo^{1 \otimes [p]} & & \dTo_{1 \otimes 1 \otimes [p]} & &
\dTo_{1} \\
V \otimes A_n & \rTo_{\rho \otimes 1} & V \otimes A_n \otimes A_n
& \rTo_{1 \otimes [p] \otimes 1} & V \otimes A \otimes A \\
\end{diagram}
Commutativity of the outermost rectangle is the assertion that the
map $V \stackrel{\rho}{\longrightarrow} V \otimes A_n \stackrel{1
\otimes [p]}{\longrightarrow} V \otimes A_n$ is a valid
$A_n$-comodule structure on $V$, and is what we are trying to
prove. Commutativity of the top left square follows since $\rho$
is a comodule map, and commutativity of the bottom left and bottom
right squares is obvious.  What remains to check is the top right
square, i.e.
\begin{diagram}
A_n & \rTo^{[p]} & A_n \\
\dTo^{\Delta} & & \dTo_{\Delta} \\
A_n \otimes A_n &  \rTo_{[p] \otimes [p]} & A_n \otimes A_n \\
\end{diagram}
which the reader can verify by hand.

\end{proof}

\begin{theorem}
\label{proofOfSufficiencyTheorem} Suppose that $p \geq
\text{max}(n,d)$, and let $k$ be a field of characteristic $p$.
Let $\phi_0, \ldots, \phi_m$ be a collection of Lie algebra
representations
 $\liefont{u}_n \rightarrow \liefont{gl}_d$ such that
\begin{enumerate}
\item{$\phi_i(X)$ is a nilpotent matrix for all $i$ and all $X \in
\liefont{u}_n$}

\item{For all $i \neq j$ and $X,Y \in \liefont{u}_n$, $\phi_i(X)$
commutes with $\phi_j(Y)$}

\end{enumerate}
For $g \in U_n$ and $0 \leq i \leq m$, define
\[ \Phi_i(g) = e^{\phi_i(\log(g))} \]
Then the formula
\[ \Phi(g) = \Phi_0(g) \Phi_1(g)^{[p]} \ldots \Phi_m(g)^{[p^m]} \]
defines a representation of $U_n$ over $k$.

\end{theorem}

\begin{proof}

Theorem \ref{theCheckListTheorem} guarantees that each $\Phi_i$ is
a representation, and the previous lemma says that so is
$\Phi_i(g)^{[p]}$.  As $\log(g)$ is an element of $\liefont{u}_n$
and as $\phi_i,\phi_j$ commute when $i \neq j$, so do $\Phi_i(g)$
and $\Phi_j(g)$, and hence so do $\Phi_i(g)^{[p^i]}$ and
$\Phi_j(g)^{[p^j]}$.  Any commuting product of representations of
an algebraic group is again a representation, whence $\Phi$ is a
representation of $U_n$.
\end{proof}

Part (1) of theorem \ref{TheMainTheorem} is now proved.  The main
theorem of this paper is now proved.

\section{Analogies with direct products in characteristic zero}

The representation theory of direct products of algebraic groups,
over any field, can be summed up as follows.

\begin{theorem}
Let $k$ be any field, $G,H$ algebraic groups over $k$, $A$ the
representing Hopf algebra of $G$, $B$ the representing Hopf
algebra of $H$.

\begin{enumerate}

\item{Let $(c_{ij})$ be the matrix formula for a representation of
$G \times H$, with $c_{ij} \in A \otimes B$, and denote by
$(a_{ij})$ and $(b_{ij})$ the induced representations on $G$ and
$H$ respectively via the canonical embeddings $G,H \rightarrow G
\times H$.  Then $(c_{ij})$ can be factored into the commuting
product
\[ (a_{ij} \otimes 1)(1 \otimes b_{ij}) \]
}

\item{Any commuting pair of representations for $G$ and $H$ yields
a representation for $G \times H$ according to the above formula.}

\end{enumerate}

Further, $\phi:V \rightarrow W$ is a morphism for the
representations $V$ and $W$ for $G \times H$ if and only if it is
both a morphism between $V$ and $W$ restricted to $G$, and for $V$
and $W$ restricted to $H$.

\end{theorem}

\begin{proof}
See for example chapter 11 of \cite{MyDissertation}.
\end{proof}

On the other hand, we have

\begin{theorem}
Let $k$ be a field of characteristic $p$, and suppose $p \geq
\text{max}(n,2d)$.  Let $\Phi, \Psi$ be
 representations of $U_n$ on the vector spaces
$V,W$ respectively, of dimension no greater than $d$, and so
necessarily of the form
\[ \Phi(g) = \Phi_0(g) \Phi_1(g) \ldots \Phi_m(g) \]
\[ \Psi(g) = \Psi_0(g) \Psi_1(g) \ldots \Psi_m(g) \]
as in theorem \ref{TheMainTheorem}.  Then the linear map $\phi:V
\rightarrow W$ is a morphism between $\Phi$ and $\Psi$ if and only
if it is a morphism between $\Phi_0$ and $\Psi_0$, between
$\Phi_1$ and $\Psi_1$, $\ldots$, and between $\Phi_m$ and
$\Psi_m$.
\end{theorem}

\begin{proof}
See theorem 10.1.1 of \cite{MyDissertation}.
\end{proof}

Thus we see that our analogy between the large characteristic
representation theory of $U_n$ and the representation theory of
$U_n^\infty$ in characteristic zero is far from superficial.

We can say even more. Denote by $\text{Rep}_k U_n$ the category of
finite dimensional representations of the algebraic group $U_n$
over the field $k$.  If $k_i$ is any sequence of fields of
strictly increasing positive characteristic, denote by $\prod_{H}
\text{Rep}_{k_i} U_n$ the `height-restricted ultraproduct' of the
categories $\text{Rep}_{k_i} U_n$ (see section 7.2 and chapter 14
of \cite{MyDissertation}). In the author's dissertation it is
proven that $\prod_{H} \text{Rep}_{k_i} U_n$ is itself the
category representations for some affine group scheme over some
field, and in fact

\begin{theorem}
\label{hprodTheorem} Let $k_i$ be a sequence of fields of strictly
increasing positive. Then $\prod_H \text{Rep}_{k_i} U_n$ is
tensorially equivalent, as neutral tannakian categories, to
$\text{Rep}_{\prod k_i} U_n^\infty$, where $U_n^\infty$ denotes a
countable direct product of copies of $U_n$, and $\prod k_i$ is
the ultraproduct of the fields $k_i$ (which is necessarily a field
of characteristic zero).

\end{theorem}

\begin{proof}
This is proven for the case of $U_3$ (there called $H_1$) in
theorem 14.0.6 of \cite{MyDissertation}.  Given the main theorem
of the present paper, the reader should be able to convince
himself that the proof given there applies equally well to $U_n$
for any $n$.
\end{proof}

Finally, we have the following generic cohomology result.

\begin{theorem}
Let $G = U_n$, let $h \in \mathbb{N}$, and let $M$ and $N$ be
modules for $G$ over $\mathbb{Z}$. Suppose that the computation
$\text{dim} \text{Ext}^{1,h}_{G^h(k)}(M,N) = m$ (height-restricted
generic cohomology; see definition 15.2.1 of
\cite{MyDissertation}) is both finite and the same for any
characteristic zero field $k$. Then if $k_i$ is any sequence of
fields of increasing positive characteristic, the computation
\[ \text{dim} \text{Ext}^{1,h}_{G(k_i)}(M,N) \]
stabilizes to $m$ for large enough $i$.

\end{theorem}

\begin{proof}
See the proof of theorem 15.2.3 of \cite{MyDissertation}, where
again it is proven for the case of $U_2$ and $U_3$, but with
theorem \ref{hprodTheorem} in hand applies equally well to every
$U_n$.

\end{proof}

Generic cohomology results for $\text{Ext}^n$, $n>1$, have been so
far elusive; we hope they will be forthcoming in the future (see
section 15.3 of \cite{MyDissertation}, and please contact the
author if you have any ideas about it ;)).

\section{Further Directions}

Our proof of part (2) of theorem \ref{TheMainTheorem} relied
essentially on the apparent nice and orderly nature of the
representing Hopf algebra $A_n$ of $U_n$; on the other hand, the
proof of part (1) given in section \ref{sufficiencySection}, so
far as we can tell, did not, and seems to apply equally well to
almost any unipotent algebraic group.  The author's current
knowledge of arbitrary unipotent algebraic groups is at the moment
lacking, so he hesitates to make any bold claim concerning this;
but it is certainly worth pursuing.

Assuming, then, that the arguments used to prove part (1) theorem
\ref{TheMainTheorem} apply equally well to an arbitrary algebraic
group, should not there exist an argument to prove (some version
of) part (2) as well?  Theorem \ref{hprodTheorem} makes it clear
that, while on the surface theorem \ref{TheMainTheorem} is
explicitly about the internal structure of certain $U_n$-modules,
at heart it is really a \emph{categorical} result; that is, it is
as much a statement about the ambient category $\text{Rep}_{k}
U_n$ as it is about the internal structure of any particular
$U_n$-module. This leads us to believe that a purely
categorical/model theoretic proof of theorem \ref{hprodTheorem}
should exist, one which would hopefully apply equally well to an
arbitrary unipotent algebraic group, and would give us all we
really need as far as questions of large scale cohomology are
concerned.  We hope that such insights will be forthcoming in a
later paper.

\section{Acknowledgements}

The author would like to sincerely thank his thesis advisor Paul R
Hewitt, under whose direction and advice theorem
\ref{heisenbergGroupTheorem} originally appeared, which was the
primary motivation for this paper.  Thanks also to Dave Hemmer for
his thoughtful advice and encouragement, and to Chris Bendel for
his kind reading of a previous version of this paper and his
thoughtful suggestions.

\bibliography{common}

\begin{thebibliography}{1}

\bibitem{MyDissertation}
Michael Crumley.
\newblock {\em Ultraproducts of Tannakian Categories and Generic Representation
  Theory of Unipotent Algebraic Groups}.
\newblock PhD thesis, The University of Toledo, Department of Mathematics,
  2010.

\bibitem{myPaper}
Michael Crumley.
\newblock Generic representation theory of the heisenberg group.
\newblock {\em arxiv.org}, May 2011.

\bibitem{hopfalgebras}
Nastasescu Dascalescu and Raianu.
\newblock {\em Hopf Algebras: An Introduction}.
\newblock Pure and Applied Mathematics. Marcel Dekker, New York, 2001.

\bibitem{dynkin3}
E.~B. Dynkin.
\newblock On the representation of the series $\log(e^x e^y)$ for
  non-commutative $x$ and $y$ by commutators.
\newblock {\em Mat.~Sbornik}, 25(67):155--162, 1949.

\bibitem{dynkin1}
E.~B. Dynkin.
\newblock {\em Lie Groups}, chapter Normed Lie Algebras and Analytic Groups,
  pages 481--485.
\newblock Translations, Series One. American Mathematical Society, Providence,
  Rhode Island, 1962.

\bibitem{dynkin2}
E.~B. Dynkin.
\newblock {\em Selected Papers of E. B. Dynkin}, chapter Calculation of the
  coefficients in the Campbell-Hausdorff formula, pages 31--35.
\newblock American Mathematical Society, Providence, Rhode Island, 2000.

\bibitem{SFB}
A~Suslin E~M~Friedlander and C~P Bendel.
\newblock Infinitesimal 1-parameter subgroups and cohomology.
\newblock {\em Journal of the AMS}, 10(3):693--728, July 1997.

\bibitem{jantzen}
Jens~Carsten Jantzen.
\newblock {\em Representations of Algebraic Groups}, volume 131 of {\em Pure
  and Applied Mathematics}.
\newblock Academic Press, Orlando, FL, 1987.

\bibitem{waterhouse}
William~C. Waterhouse.
\newblock {\em Introduction to Affine Group Schemes}.
\newblock Graduate Texts in Mathematics. Springer-Verlag, New York, 1979.

\end{thebibliography}
\bibliographystyle{plain}

\end{document}